%% file: maxconst2.tex
\documentclass[12pt,a4paper,twoside]{article}

\title{\sc On the Maxwell Constants in 3D}
\def\shorttitle{On the Maxwell Constants in 3D}
\def\pauthor{Dirk Pauly}

\def\mylabelonoff{off}
\def\allowdisbrk{no}

\input head_paper

\input head_paule

\usepackage{nicefrac}

\def\cn{\mathbb{C}}
\def\epsch{\epsh^{-1}}

\def\X{\mathsf{X}}
\def\Y{\mathsf{Y}}
\def\Z{\mathsf{Z}}
\def\H{\mathsf{H}}
\def\M{\mathrm{M}}

\def\L{\mathrm{L}}
\def\A{\mathrm{A}}
\def\As{\A^{*}}
\def\B{\mathrm{B}}
\def\Bs{\B^{*}}
\def\T{\mathrm{T}}

\def\cA{\mathcal{A}}
\def\cAs{\cA^{*}}
\def\cM{\mathcal{M}}

\def\gad{\ga_{\!\mathtt{t}}}
\def\gan{\ga_{\!\mathtt{n}}}

\def\xt{\tilde{x}}
\def\xh{\hat{x}}
\def\yt{\tilde{y}}
\def\zt{\tilde{z}}

\def\spec{\sigma}
\def\resol{\rho}
\def\specp{\sigma_{\mathtt{p}}}
\def\specc{\sigma_{\mathtt{c}}}
\def\specr{\sigma_{\mathtt{r}}}

\renewcommand{\norm}[1]{\normos{#1}}
\newcommand{\normx}[1]{\norm{#1}_{\X}}
\newcommand{\normy}[1]{\norm{#1}_{\Y}}
\newcommand{\normz}[1]{\norm{#1}_{\Z}}
\newcommand{\normltom}[1]{\norm{#1}_{\ltom}}
\newcommand{\normltepsom}[1]{\norm{#1}_{\ltepsom}}
\newcommand{\normltepsmoom}[1]{\norm{#1}_{\ltepsmoom}}
\newcommand{\normltmuom}[1]{\norm{#1}_{\ltmuom}}
\newcommand{\normltmumoom}[1]{\norm{#1}_{\ltmumoom}}

\renewcommand{\scp}[2]{\scps{#1}{#2}}
\newcommand{\scpx}[2]{\scp{#1}{#2}_{\X}}
\newcommand{\scpy}[2]{\scp{#1}{#2}_{\Y}}

\newcommand{\scpltom}[2]{\scp{#1}{#2}_{\ltom}}
\newcommand{\scpltepsom}[2]{\scp{#1}{#2}_{\ltepsom}}

\newcommand{\scpltqom}[2]{\scp{#1}{#2}_{\ltqom}}
\newcommand{\scpltqpoom}[2]{\scp{#1}{#2}_{\ltqpoom}}

\renewcommand{\Cgen}[2]{\mathsf{C}^{#1}_{#2}}
\def\cirt{\Cgen{\infty}{}(\rt)}
\def\cigadom{\Cgen{\infty}{\gad}(\om)}

\def\ltom{\Ltom}
\def\ltepsom{\Ltepsom}
\def\ltepsmoom{\Ltepsmoom}
\def\ltmuom{\Ltmuom}
\def\ltmumoom{\Ltmumoom}

\def\hogaom{\Hgen{1}{\ga}{}(\om)}
\def\hogadom{\Hgen{1}{\gad}{}(\om)}

\def\hoemptyom{\Hgen{1}{\emptyset}{}(\om)}

\def\dgaom{\divgen{\ga}{}{}(\om)}
\def\dgadom{\divgen{\gad}{}{}(\om)}
\def\dganom{\divgen{\gan}{}{}(\om)}
\def\dgazom{\divgen{\ga,0}{}{}(\om)}
\def\dgadzom{\divgen{\gad,0}{}{}(\om)}
\def\dganzom{\divgen{\gan,0}{}{}(\om)}

\def\rgaom{\rotgen{\ga}{}{}(\om)}
\def\rgadom{\rotgen{\gad}{}{}(\om)}
\def\rganom{\rotgen{\gan}{}{}(\om)}

\def\rgadzom{\rotgen{\gad,0}{}{}(\om)}
\def\rganzom{\rotgen{\gan,0}{}{}(\om)}

\def\ltqom{\Ltqom}
\def\ltqepsom{\Lgen{2,q}{\eps}(\om)}
\def\ltqepsmoom{\Lgen{2,q}{\epsmo}(\om)}

\def\ltqpoom{\Ltqpoom}
\def\ltqpomuom{\Lgen{2,q+1}{\mu}(\om)}
\def\ltqpomumoom{\Lgen{2,q+1}{\mumo}(\om)}
\def\ltqmoom{\Ltqmoom}

\def\dqom{\Dqom}
\def\dqgaom{\Dgen{q}{\ga}{}(\om)}
\def\dqgadom{\Dgen{q}{\gad}{}(\om)}
\def\dqgadzom{\Dgen{q}{\gad,0}{}(\om)}
\def\dqpogadzom{\Dgen{q+1}{\gad,0}{}(\om)}

\def\deqom{\Deqom}
\def\deqpoom{\Deqpoom}
\def\deqganom{\Degen{q}{\gan}{}(\om)}
\def\deqpoganom{\Degen{q+1}{\gan}{}(\om)}
\def\deqganzom{\Degen{q}{\gan,0}{}(\om)}
\def\deqpoganzom{\Degen{q+1}{\gan,0}{}(\om)}

\newcommand{\harmqdnom}{\harmom{q}{\mathtt{DN}}}
\newcommand{\harmqpodnom}{\harmom{q+1}{\mathtt{DN}}}

\DeclareMathOperator{\rotcalspace}{\mathcal{R}}
\newcommand{\rotcalgen}[2]{\overset{#2}{\rotcalspace}{}_{#1}}
\def\rcalgadom{\rotcalgen{\gad}{}{}(\om)}
\def\rcalganom{\rotcalgen{\gan}{}{}(\om)}

\def\ca{c_{\A}}
\def\cas{c_{\As}}
\def\cm{c_{\M}}

\def\cml{c_{\M,\lambda}}
\def\cpga{c_{\mathtt{p},\ga}}
\def\cpgad{c_{\mathtt{p},\gad}}
\def\cpempty{c_{\mathtt{p},\emptyset}}
\def\cpgaeps{c_{\mathtt{p},\ga,\eps}}

\def\cpgadid{c_{\mathtt{p},\gad,\id}}
\def\cpgadeps{c_{\mathtt{p},\gad,\eps}}
\def\cpemptyeps{c_{\mathtt{p},\emptyset,\eps}}

\def\cmga{c_{\mathtt{m},\ga}}
\def\cmempty{c_{\mathtt{m},\emptyset}}
\def\cmgad{c_{\mathtt{m},\gad}}

\def\cmgaeps{c_{\mathtt{m},\ga,\eps}}

\def\cmemptyeps{c_{\mathtt{m},\emptyset,\eps}}

\def\cmgadid{c_{\mathtt{m},\gad,\id}}
\def\cmgadeps{c_{\mathtt{m},\gad,\eps}}

\def\cmgadiveps{c_{\mathtt{m},\ga,\div,\eps}}

\def\cmgandiveps{c_{\mathtt{m},\gan,\div,\eps}}
\def\cmemptydiveps{c_{\mathtt{m},\emptyset,\div,\eps}}

\def\cmgadrotepsmu{c_{\mathtt{m},\gad,\rot,\eps,\mu}}

\def\cmganrotmueps{c_{\mathtt{m},\gan,\rot,\mu,\eps}}

\def\cmgadrotepseps{c_{\mathtt{m},\gad,\rot,\eps,\eps}}
\def\cmganrotepseps{c_{\mathtt{m},\gan,\rot,\eps,\eps}}

\def\cmgarotepsid{c_{\mathtt{m},\ga,\rot,\eps,\id}}
\def\cmgadrotepsid{c_{\mathtt{m},\gad,\rot,\eps,\id}}
\def\cmganrotepsid{c_{\mathtt{m},\gan,\rot,\eps,\id}}
\def\cmemptyrotepsid{c_{\mathtt{m},\emptyset,\rot,\eps,\id}}

\def\cmgarotepsid{c_{\mathtt{m},\ga,\rot,\eps,\id}}
\def\cmgarotidid{c_{\mathtt{m},\ga,\rot,\id,\id}}
\def\cmgadrotidid{c_{\mathtt{m},\gad,\rot,\id,\id}}
\def\cmganrotidid{c_{\mathtt{m},\gan,\rot,\id,\id}}

\def\cmganrotideps{c_{\mathtt{m},\gan,\rot,\id,\eps}}

\def\cmemptyrotepsid{c_{\mathtt{m},\emptyset,\rot,\eps,\id}}
\def\cmemptyrotidid{c_{\mathtt{m},\emptyset,\rot,\id,\id}}
\def\cmgadrotepsmu{c_{\mathtt{m},\gad,\rot,\eps,\mu}}

\def\cmgadrot{c_{\mathtt{m},\gad,\rot}}
\def\cmganrot{c_{\mathtt{m},\gan,\rot}}
\def\cmgarot{c_{\mathtt{m},\ga,\rot}}
\def\cmemptyrot{c_{\mathtt{m},\emptyset,\rot}}
\def\cmroteps{c_{\mathtt{m},\rot,\eps}}
\def\cmgaroteps{c_{\mathtt{m},\ga,\rot,\eps}}
\def\cmemptyroteps{c_{\mathtt{m},\emptyset,\rot,\eps}}
\def\cmgadroteps{c_{\mathtt{m},\gad,\rot,\eps}}
\def\cmganroteps{c_{\mathtt{m},\gan,\rot,\eps}}

\begin{document}

\maketitle{}

\begin{center}
{\tt Dedicated to
Martin Costabel\\
on the occasion of his 65th birthday}
\end{center}

\begin{abstract}
Using tools from functional analysis
we show that for bounded and convex domains in three dimensions, 
the Maxwell constants are bounded from below and above 
by Friedrichs' and Poincar\'e's constants.\\
\end{abstract}

\tableofcontents

\section{Introduction and Preliminaries}

Throughout this paper, let us fix a bounded domain $\om\subset\rt$
with boundary $\ga:=\pom$, which is devided into 
two relatively open subsets $\gad$ and its complement $\gan:=\ga\setminus\ol{\gad}$. 
The letters \texttt{t} and \texttt{n} should remind on
homogeneous tangential and normal boundary conditions.

It is well known that the Poincar\'e (or Friedrichs) inequality, i.e.,
for all $u\in\hogadom$
\begin{align}
\mylabel{introestgrad}
\normltom{u}
\leq\cpgadeps\normltepsom{\na u},
\end{align}
holds with some $\cpgadeps>0$, as long as Rellich's selection theorem is valid, i.e., the embedding 
\begin{align}
\mylabel{introgradcomp}
\hoom\hookrightarrow\ltom
\end{align}
is compact. Here, $\ltom$ and $\hoom$ denote the usual Lebesgue- and Sobolev (Hilbert) spaces,
respectively. Moreover, $\eps:\om\to\rttt$ denotes a symmetric and uniformly positive definite
$\Li$-matrix field. We introduce $\ltepsom$ as $\ltom$ equipped with the weighted inner product 
$\scpltepsom{\,\cdot\,}{\,\cdot\,}:=\scpltom{\eps\,\cdot\,}{\,\cdot\,}$.\footnote{Throughout
this paper norms resp.~scalar products will be denoted 
by $\normx{\,\cdot\,}$ resp.~$\scpx{\,\cdot\,}{\,\cdot\,}$
if $\X$ is a normed space or a space featuring a scalar product.}
For $\gad\neq\emptyset$ the Sobolev space 
$\hogadom$ is defined as the closure \big(taken in $\hoom$\big) of test functions
$$\cigadom:=\set{\varphi|_{\om}}{\varphi\in\cirt\,,\,\dist(\supp\varphi,\gad)>0}.$$
Otherwise we set $\hoemptyom:=\hoom\cap\rz^{\bot}$.
Let us assume that we have chosen the best constant in \eqref{introestgrad}, this is
$$\frac{1}{\cpgadeps}:=\inf_{0\neq u\in\hogadom}\frac{\normltepsom{\na u}}{\normltom{u}}.$$

Analogously, it is also well known that the (let's call it) Maxwell inequality, i.e.,
for all $E\in\rgadom\cap\epsmo\dganom$
\begin{align}
\nonumber
\normltepsom{E-\pidn E}
&\leq\cmgadeps\big(\normltom{\div\eps E}^2+\normltom{\rot E}^2\big)^{\nicefrac{1}{2}}
\intertext{or equivalently for all $E\in\rgadom\cap\epsmo\dganom\cap\harmdnepsom^{\bot_{\eps}}$}
\mylabel{intro-estrotdiv}
\normltepsom{E}
&\leq\cmgadeps\big(\normltom{\div\eps E}^2+\normltom{\rot E}^2\big)^{\nicefrac{1}{2}},
\end{align}
holds with some $\cmgadeps>0$, as long as the Maxwell selection theorem or
the Maxwell compactness property is given, i.e., the embedding 
\begin{align}
\mylabel{introrotdivcomp}
\rgadom\cap\epsmo\dganom\hookrightarrow\ltom
\end{align}
is compact, see Appendix \ref{appmaxest} for details. 
Here, we introduce the Sobolev (Hilbert) spaces
$$\rom:=\set{E\in\ltom}{\rot E\in\ltom},\quad
\dom:=\set{E\in\ltom}{\div E\in\ltom}$$
in the distributional sense. As above, if $\gad\neq\emptyset$,
we define as closures \big(taken in $\rom$ resp.~$\dom$\big) of test vector fields $\cigadom$
the Sobolev spaces $\rgadom$ and $\dgadom$ (and of course the same for $\gan$).
If $\gad=\emptyset$ we set $\rotgen{\emptyset}{}(\om):=\rom$ and $\divgen{\emptyset}{}(\om):=\dom$.
Then, for $\gad\neq\emptyset$
in $\hogadom$, $\rgadom$ and $\dgadom$ homogeneous scalar, tangential and normal traces
at $\gad$ are generalized, respectively. Moreover, we define the closed subspaces
$$\rzom:=\set{E\in\ltom}{\rot E=0},\quad
\dzom:=\set{E\in\ltom}{\div E=0}$$
as well as $\rgadzom:=\rgadom\cap\rzom$ and $\dgadzom:=\dgadom\cap\dzom$.
Finally, we have the harmonic Dirichlet-Neumann fields
$$\harmdnepsom:=\rgadzom\cap\epsmo\dganzom,$$
which are finite dimensional since by \eqref{introrotdivcomp} 
the unit ball is compact in $\harmdnepsom$.
The $\ltepsom$-orthogonal projector onto them 
will be denoted by $\pidn:\ltepsom\to\harmdnepsom$
and $\bot_{\eps}$ means orthogonality in $\ltepsom$.
If $\gad=\ga$ resp.~$\gan=\ga$ we have the classical Dirichlet resp.~Neumann fields 
and write $\harmdepsom$ resp.~$\harmnepsom$.
We also need the Neumann-Dirichlet fields 
$\harmndepsom:=\rganzom\cap\epsmo\dgadzom$.
In the case $\eps=\id$ we usually omit $\eps$ in our notations.
Again, we assume that also in \eqref{intro-estrotdiv} the best constant
$$\frac{1}{\cmgadeps}:=\inf_{0\neq E\in\rgadom\cap\epsmo\dganom\cap\harmdnepsom^{\bot_{\eps}}}
\frac{\big(\normltom{\div\eps E}^2+\normltom{\rot E}^2\big)^{\nicefrac{1}{2}}}{\normltepsom{E}}$$
is taken.

The crucial property for \eqref{intro-estrotdiv} to hold
is the Maxwell compactness property \eqref{introrotdivcomp},
which holds, e.g., if $\om$ has a (strongly) Lipschitz continuous boundary $\ga$
with a (strongly) Lipschitz continuous interface $\gamma:=\ol{\gad}\cap\ol{\gan}$,
see \cite{jochmanncompembmaxmixbc} for details.
More precisely, the boundary $\ga$ and the interface $\gamma$ 
can be described locally as graphs of Lipschitz functions.
From now on we assume this properties of $\ga$ and $\gad$, $\gan$
as general assumption. Note that then also \eqref{introgradcomp}
and \eqref{introestgrad} hold.
Another successful approach proving the Maxwell compactness property 
using a different technique from \cite{weckmax} has been shown in \cite{kuhndiss}.
For the Maxwell compactness property in the case of full boundary conditions we refer to 
\cite{weckmax,picardpotential,picardboundaryelectro,picardcomimb,webercompmax,leisbook,costabelremmaxlip,picardweckwitschxmas,saranenmaxkegel,saranenmaxnichtglatt,saranenineqfried,witschremmax}.

With the help of the $\ltepsom$-orthogonal Helmholtz decomposition
\begin{align}
\mylabel{introhelmdeco}
\ltepsom
&=\na\hogadom\oplus_{\eps}\harmdnepsom\oplus_{\eps}\epsmo\rot\rganom,
\intertext{where}
\nonumber
\rgadzom
&=\na\hogadom\oplus_{\eps}\harmdnepsom,\quad
\epsmo\dganzom
=\epsmo\rot\rganom\oplus_{\eps}\harmdnepsom,
\end{align}
see Appendix \ref{appmaxhelm} for details, 
we can split the estimate \eqref{intro-estrotdiv} into two, namely
\begin{align}
\mylabel{}
\forall\,E&\in\epsmo\dganom\cap\na\hogadom&
\normltepsom{E}&\leq\cmgandiveps\normltom{\div\eps E},\\
\forall\,E&\in\rgadom\cap\epsmo\rot\rganom&
\normltepsom{E}&\leq\cmgadrotepsid\normltom{\rot E},
\end{align}
where we again assume to use the best constants
\begin{align*}
\frac{1}{\cmgandiveps}&:=\inf_{0\neq E\in\epsmo\dganom\cap\na\hogadom}
\frac{\normltom{\div\eps E}}{\normltepsom{E}},\\
\frac{1}{\cmgadrotepsid}&:=\inf_{0\neq E\in\rgadom\cap\epsmo\rot\rganom}
\frac{\normltom{\rot E}}{\normltepsom{E}}.
\end{align*}

By the assumptions on $\eps$ there exist $\epsu,\epso>0$
such that for all $E\in\ltom$
\begin{align*}
\frac{1}{\epsu}\normltom{E}
&\leq\normltepsom{E}
\leq\epso\normltom{E}.
\intertext{We note $\normltepsom{E}=\normltom{\eps^{1/2}E}$
and $\normltepsom{\eps^{1/2}E}=\normltom{\eps E}$. Thus, for all $E\in\ltom$}
\frac{1}{\epsu}\normltepsom{E}
&\leq\normltom{\eps E}
\leq\epso\normltepsom{E}.
\end{align*}
The inverse $\epsmo$ satisfies for all $E\in\ltom$
$$\frac{1}{\epso}\normltom{E}
\leq\normltepsmoom{E}
\leq\epsu\normltom{E},\quad
\frac{1}{\epso}\normltepsmoom{E}
\leq\normltom{\epsmo E}
\leq\epsu\normltepsmoom{E},$$
which immediately follows by
$$\normltepsmoom{E}=\normltom{\eps^{-1/2}E}
\begin{cases}
\leq\epsu\normltepsom{\eps^{-1/2}E}
=\epsu\normosom{E}\\
\geq\epso^{-1}\normltepsom{\eps^{-1/2}E}
=\epso^{-1}\normosom{E}
\end{cases}.$$
For later purposes let us also define $\epsh:=\max\{\epsu,\epso\}$.

In this contribution we will study these different constants
$\cpgadeps$, $\cmgadeps$, $\cmgandiveps$, $\cmgadrotepsid$
and their relations to each other.
It turns out that 
$$\cpgadeps=\cmgandiveps,\quad
\cmgadrotepsid=\cmganrotideps,\quad
\cmgadeps=\max\{\cpgadeps,\cmgadrotepsid\}$$ 
hold, see Lemmas \ref{estgraddivcpdiv}, \ref{maintheononconvexlem} and \ref{estrotrotcmax}.
The main result of this paper states that in the special case 
of full boundary conditions, i.e., $\gad=\ga$ or $\gan=\ga$,
and for bounded and \emph{convex} domains we have
$$\frac{\cpga}{\epso}
\leq\cmgaeps
\leq\epsh\cp,\quad
\frac{\cp}{\epso}
\leq\cmemptyeps
\leq\epsh\cp$$
and especially for $\eps=\id$
$$\max\{\cpga,\cmrot\}=\cmga\leq\cmempty=\cp,$$
see Theorem \ref{maintheoconvexfullboundary}.
Here, we introduce for the special case $\eps=\id$
$$\cpgad:=\cpgadid,\quad
\cp:=\cpempty,\quad
\cmgad:=\cmgadid$$
and
$$\cmgadrot:=\cmgadrotidid=\cmganrotidid=\cmganrot$$
as well as
$$\cmrot:=\cmgarotidid=\cmemptyrotidid.$$

The crucial point in our analysis is that for convex domains
$$\cmrot\leq\cp,\quad
\cmgarotepsid,\cmemptyrotepsid\leq\epso\cp$$
hold, see Lemma \ref{mainlem}.
Some of these results have also been obtained recently in \cite{paulymaxconst1}
utilizing different and more elementary\footnote{In the sense that
no tools from functional analysis were used.} methods.
We note that in the convex case we can estimate 
the Poincar\'e constant $\cp$ by the diameter of $\om$.
More precisely, by the famous paper of Payne and Weinberger 
\cite{payneweinbergerpoincareconvex}\footnote{A little mistake 
or inconsistency in \cite{payneweinbergerpoincareconvex} 
has been corrected later in \cite{bebendorfpoincareconvex}.} we have 
$$\cp\leq\frac{\diam(\om)}{\pi}.$$
In \cite{payneweinbergerpoincareconvex} also the optimality of this estimate has been shown.
Furthermore, $\cpga<\cp$ is well known even for non-convex domains, 
see e.g.~\cite{filonovdirneulapeigen} and the cited literature, yielding
\begin{align}
\mylabel{friedrichspoincareconstest}
\frac{1}{\sqrt{\lambda_{1}}}
=\cpga
<\cp
=\frac{1}{\sqrt{\mu_{2}}}
\leq\frac{\diam(\om)}{\pi},
\end{align}
where $\lambda_{1}$ resp.~$\mu_{2}$
is the first Dirichlet resp.~second Neumann eigenvalue of the negative Laplacian.

At least some of our results extend in a natural way to bounded domains 
$\om\subset\rN$ or even to Riemannian manifolds with compact closure,
see Remark \ref{estgraddivrN} and Appendix \ref{appgenop}.

Our new estimates have important applications e.g.~to numerical analysis,
where especially an upper bound for the Maxwell constants 
is needed e.g.~for preconditioning and for functional a posteriori error estimates
in the framework of Maxwell's equations.

\section{An Abstract Setting}
\mylabel{fa}

Let $\X$ and $\Y$ be Hilbert spaces and 
$$\A:D(\A)\subset\X\to\Y,\quad
\As:D(\As)\subset\Y\to\X$$ 
be a closed and densely defined linear operator and its adjoint.
Here, $D$ denotes the domain of definition
and we introduce the kernel $N$ and the range $R$.
Since $\A$ is closed we have $(\As)^{*}=\bar{\A}=\A$ 
and sometimes $(\A,\As)$ is called a dual pair.
The projection theorem yields the orthogonal `Helmholtz' decompositions 
\begin{align}
\mylabel{genortho}
\X=N(\A)\oplus\ol{R(\As)},\quad
\Y=N(\As)\oplus\ol{R(\A)}.
\end{align}
Now, we collect some well known facts. 
For the convenience of the reader we give simple proofs of those 
in the Appendix \ref{appfa}.

$\As\A$ and $\A\As$ are non-negative and self-adjoint and their spectra coincide 
if we exclude $\{0\}$, i.e.,
\begin{align}
\mylabel{spectrumAAs}
\sigma(\As\A)\setminus\{0\}
=\sigma(\A\As)\setminus\{0\},\quad
\sigmap(\As\A)\setminus\{0\}
=\sigmap(\A\As)\setminus\{0\}.
\end{align}
Let us assume that the embedding 
\begin{align}
\mylabel{compactembA}
D(\A)\cap\ol{R(\As)}\hookrightarrow\X
\end{align}
is compact.

\begin{lem}
\mylabel{falemone}
There exist $\ca,\cas>0$, such that
\begin{align*}
\forall\,x&\in D(\A)\cap R(\As)&\normx{x}&\leq\ca\normy{\A x},\\
\forall\,y&\in D(\As)\cap R(\A)&\normy{y}&\leq\cas\normx{\As y}.
\end{align*}
Moreover, $R(\A)$ and $R(\As)$ are closed and
$$\X=N(\A)\oplus R(\As),\quad
\Y=N(\As)\oplus R(\A).$$
Furthermore, $D(\As)\cap R(\A)\hookrightarrow\Y$ is compact as well.
\end{lem}

We note that the same lemma can be proved assuming the compactness of the embedding
of $D(\As)\cap\ol{R(\A)}\hookrightarrow\Y$ instead of \eqref{compactembA}.
By Lemma \ref{falemone} the restricted operator
$$\cA:=\A|_{D(\cA)}:D(\cA)\subset R(\As)\to R(\A),\quad
D(\cA):=D(\A)\cap R(\As)$$
has a bounded inverse $\cA^{-1}:R(\A)\to D(\cA)$ with $\norm{\cA^{-1}}\leq(1+\ca^2)^{\nicefrac{1}{2}}$,
which is compact as an operator from $R(\A)$ to $R(\As)$.
Hence, $\As\A$ and $\A\As$ have pure point spectra which can only accumulate at infinity
and which coincide by \eqref{spectrumAAs}.
Especially, the second eigenvalues equal and therefore
(see Corollary \ref{eigenvaluesmin} for details) we conclude:

\begin{theo}
\mylabel{fatheothree}
For the best constants in Lemma \ref{falemone} it holds $\ca=\cas$, this is
$$\frac{1}{\ca}=\min_{0\neq x\in D(\A)\cap R(\As)}\frac{\normy{\A x}}{\normx{x}}
=\min_{0\neq y\in D(\As)\cap R(\A)}\frac{\normx{\As y}}{\normy{y}}=\frac{1}{\cas}.$$
\end{theo}

Hence, $\ca^{-2}=\cas^{-2}$ is the first positive eigenvalue 
of $\As\A$ as well as of $\A\As$.

\section{The Maxwell Estimates}

We remind on $\om$ and its properties from the introduction.

\subsection{General Lipschitz Domains}

In this subsection we frequently use Lemma \ref{falemone} and Theorem \ref{fatheothree}.

\subsubsection{Gradient and Divergence}

Let us consider $\A$ as 
$$\na:\hogadom\subset\ltom\to\ltepsom.$$
Then $\As$ equals 
$$-\div\eps:\epsmo\dganom\subset\ltepsom\to\ltom.$$
More precisely, we have the following table:
\begin{center}\begin{tabular}{|c||c|c|c||c|c|}
\hline
$\A$ & $D(\A)$ & $\X$ & $\Y$ & $N(\A)$ & $R(\A)$ \\
\hline
$\na$ & $\hogadom$ & $\ltom$ & $\ltepsom$ & $\{0\}$ & $\na\hogadom=\rgadzom\cap\harmdnom^{\bot}$\\
\hline\hline
$\As$ & $D(\As)$ & $\Y$ & $\X$ & $N(\As)$ & $R(\As)$\\
\hline
$-\div\eps$ & $\epsmo\dganom$ & $\ltepsom$ & $\ltom$ & $\epsmo\dganzom$ & $\div\dganom$\\
\hline
\end{tabular}\end{center}
We note that $\div\dganom=\ltom$ if $\gan\neq\ga$ and $\div\dgaom=\ltom\cap\rz^{\bot}$.
Moreover, we emphasize that
indeed $D(\As)=\epsmo\dganom$ holds, 
see e.g.~\cite{jochmanncompembmaxmixbc}.
Note that for this one has to show the approximation property
$$\dganom=\set{H\in\dom}{\scpltom{\div H}{u}=-\scpltom{H}{\na u}\,\forall\;u\in\hogadom},$$
which is not trivial at all for mixed boundary conditions.
Only in the special cases of full boundary conditions this is clear.
$D(\As)=\epsmo\dom$ holds for $\gad=\ga$ by definition.
For $\gad=\emptyset$ we see that the closed operator 
$$\B:=-\div:\dgaom\subset\ltom\to\ltom$$
has the adjoint
$$\Bs=\na:\hoom\subset\ltom\to\ltom$$
by definition. Since in this case $\A=\Bs$ we have
$D(\As)=D(\B^{**})=D(\B)=\dgaom$.
The crucial compact embedding \eqref{compactembA} reads
$$\hogadom\cap\ol{\div\dganom}\hookrightarrow\ltom$$
and is just Rellich's selection theorem since
$$\hogadom\cap\ol{\div\dganom}\subset\hogadom\subset\hoom\hookrightarrow\ltom.$$
Theorem \ref{fatheothree} yields
$$0<\frac{1}{\cpgadeps}
=\min_{0\neq u\in\hogadom}\frac{\normltepsom{\na u}}{\normltom{u}}
=\min_{0\neq E\in\epsmo\dganom\cap\na\hogadom}\frac{\normltom{\div\eps E}}{\normltepsom{E}}
=\frac{1}{\cmgandiveps}.$$
We note that $\lambda_{\gad,\eps}:=\cpgadeps^{-2}$ is the first positive
Dirichlet-Neumann eigenvalue of the weighted negative Laplacian $-\Delta_{\eps}:=-\div\eps\na$.
For $\eps=\id$ and $\gad=\ga$ resp.~$\gad=\emptyset$ we see that
$\lambda_{\ga,\id}=:\lambda_{1}$ resp.~$\lambda_{\emptyset,\id}=:\mu_{2}$
is the first Dirichlet resp.~second Neumann eigenvalue of the negative Laplacian.
As $\lambda_{\gad,\eps}=\cmgandiveps^{-2}$ holds too, 
$\lambda_{\gad,\eps}$ is also the first positive Neumann-Dirichlet eigenvalue 
of the weighted negative reduced grad-div-operator $-\na\div\eps$,
which can also be interpreted as the weighted negative vector Laplacian
$-\vec{\Delta}_{\eps}:=-\na\div\eps+\rot\rot$ on a subspace of irrotational vector fields.

\begin{lem}
\mylabel{estgraddivcpdiv}
The Poincar\'e constant in $\hogadom$
and the Maxwell divergence constant in $\epsmo\dganom\cap\na\hogadom$, i.e.,
the best constants in the inequalities
\begin{align*}
\forall\,u&\in\hogadom&
\normltom{u}&\leq\cpgadeps\normltepsom{\na u},\\
\forall\,E&\in\epsmo\dganom\cap\na\hogadom&
\normltepsom{E}&\leq\cmgandiveps\normltom{\div\eps E},
\end{align*}
coincide and correspond to the first positive 
Dirichlet-Neumann eigenvalue of the weighted negative Laplacian $-\Delta_{\eps}$, 
more precisely $\cpgadeps=\cmgandiveps=1/\sqrt{\lambda_{\gad,\eps}}$.
\end{lem}

\begin{lem}
\mylabel{estgraddivcpdivepsest}
It holds $\epso^{-1}\cpgad\leq\cpgadeps\leq\epsu\cpgad$
as well as $\cpga\leq\cpgad$ and $\cpgaeps\leq\cpgadeps$.
\end{lem}

\begin{proof}
For $u\in\hogadom$ we have
\begin{align*}
\normltom{u}
&\leq\cpgad\normltom{\na u}
\leq\epsu\cpgad\normltepsom{\na u},\\
\normltom{u}
&\leq\cpgadeps\normltepsom{\na u}
\leq\epso\cpgadeps\normltom{\na u},
\end{align*}
which gives $\cpgadeps\leq\epsu\cpgad$ and $\cpgad\leq\epso\cpgadeps$.
\end{proof}

\begin{rem}
\mylabel{estgraddivrN}
The results of this section extend to bounded domains $\om\subset\rN$, $N\in\nz$,
having the proper regularity of the boundary.
\end{rem}

\subsubsection{Rotations}
\mylabel{secrotations}

Now, let $\A$ be
$$\mumo\rot:\rgadom\subset\ltepsom\to\ltmuom.$$
Then $\As$ is
$$\epsmo\rot:\rganom\subset\ltmuom\to\ltepsom,$$
where $\mu$ is another matrix field similar to $\eps$.
More precisely:
\begin{center}\begin{tabular}{|c||c|c|c||c|c|}
\hline
$\A$ & $D(\A)$ & $\X$ & $\Y$ & $N(\A)$ & $R(\A)$ \\
\hline
$\mumo\rot$ & $\rgadom$ & $\ltepsom$ & $\ltmuom$ & $\rgadzom$ & $\mumo\rot\rgadom$\\
\hline\hline
$\As$ & $D(\As)$ & $\Y$ & $\X$ & $N(\As)$ & $R(\As)$\\
\hline
$\epsmo\rot$ & $\rganom$ & $\ltmuom$ & $\ltepsom$ & $\rganzom$ & $\epsmo\rot\rganom$\\
\hline
\end{tabular}\end{center}
We note
$$R(\A)=\mumo\big(\dgadzom\cap\harmndom^{\bot}\big),\quad
R(\As)=\epsmo\big(\dganzom\cap\harmdnom^{\bot}\big)$$
and that indeed $D(\As)=\rganom$ holds, 
see again e.g.~\cite{jochmanncompembmaxmixbc}.
As before, for this one has to show the approximation property
$$\rganom=\set{H\in\rom}{\scpltom{\rot H}{E}=\scpltom{H}{\rot E}\,\forall\;E\in\rgadom},$$
which is not trivial at all for mixed boundary conditions.
Again, only in the special cases of full boundary conditions this is clear.
Since $D(\As)=\rom$ holds for $\gad=\ga$ by definition
we have also $D(\Bs)=D(\A^{**})=D(\A)=\rgaom$ for $\B=\As$,
which shows the result for $\gad=\emptyset$.
The crucial compact embedding \eqref{compactembA} reads
$$\rgadom\cap\epsmo\ol{\rot\rganom}\hookrightarrow\ltepsom$$
and is just the Maxwell compactness property \eqref{introrotdivcomp} since
$$\rgadom\cap\epsmo\ol{\rot\rganom}\subset\rgadom\cap\epsmo\dganzom
\subset\rgadom\cap\epsmo\dganom\hookrightarrow\ltom\subset\ltepsom.$$
By Theorem \ref{fatheothree} we have
\begin{align*}
0<\frac{1}{\cmgadrotepsmu}
&=\min_{0\neq E\in\rgadom\cap\epsmo\rot\rganom}\frac{\normltmuom{\mumo\rot E}}{\normltepsom{E}}\\
&=\min_{0\neq H\in\rganom\cap\mumo\rot\rgadom}\frac{\normltepsom{\epsmo\rot H}}{\normltmuom{H}}
=\frac{1}{\cmganrotmueps},
\end{align*}
which serves also as definition for the constants $\cmgadrotepsmu$ and $\cmganrotmueps$.
Therefore, $\kappa_{\gad,\eps,\mu}:=\cmgadrotepsmu^{-2}$ is the first positive 
Dirichlet-Neumann eigenvalue of the weighted reduced 
double-rot-operator $\square_{\eps,\mu}:=\epsmo\rot\mumo\rot$,
which can also be interpreted as the weighted negative vector Laplacian
$-\vec{\Delta}_{\eps,\mu}:=-\na\div\eps+\epsmo\rot\mumo\rot$ 
on a subspace of $\eps$-solenoidal vector fields.
Since $\kappa_{\gad,\eps,\mu}=\cmganrotmueps^{-2}$ holds as well, 
$\kappa_{\gad,\eps,\mu}$ is also the first positive
Neumann-Dirichlet eigenvalue of the weighted reduced 
double-rot-operator $\square_{\mu,\eps}=\mumo\rot\epsmo\rot$,
which can also be interpreted as the weighted negative vector Laplacian
on a subspace of $\mu$-solenoidal vector fields, i.e.,
$-\vec{\Delta}_{\mu,\eps}=-\na\div\mu+\mumo\rot\epsmo\rot$.

\begin{lem}
\mylabel{estrotrotcmax}
The tangential-normal and normal-tangential Maxwell rotation constants, i.e.,
the best constants in the inequalities
\begin{align*}
\forall\,E&\in\rgadom\cap\epsmo\rot\rganom&
\normltepsom{E}&\leq\cmgadrotepsmu\normltmumoom{\rot E},\\
\forall\,H&\in\rganom\cap\mumo\rot\rgadom&
\normltmuom{H}&\leq\cmganrotmueps\normltepsmoom{\rot H},
\end{align*}
coincide and correspond to the first positive
Dirichlet-Neumann eigenvalue of the weighted reduced double-rot-operator $\square_{\eps,\mu}$, 
more precisely $\cmgadrotepsmu=\cmganrotmueps=1/\sqrt{\kappa_{\gad,\eps,\mu}}$.
\end{lem}

Let us define for $\eps=\mu$ and for $\eps=\mu=\id$
$$\cmgadroteps:=\cmgadrotepseps=\cmganrotepseps$$
and note
\begin{align}
\mylabel{rotconsteq}
\cmgadroteps=\cmganroteps,\quad
\cmgadrot=\cmganrot.
\end{align}

\begin{cor}
\mylabel{estrotrotcmaxcor}
For all $E\in\big(\rgadom\cap\epsmo\rot\rganom\big)\cup\big(\rganom\cap\epsmo\rot\rgadom\big)$
\begin{align}
\mylabel{estrotrotcmaxcorineqone}
\normltepsom{E}
\leq\cmgadroteps\normltepsmoom{\rot E}
\leq\epsu\cmgadroteps\normltom{\rot E}
\end{align}
holds with sharp constants.
Moreover, the inequalities 
\begin{align}
\mylabel{estrotrotcmaxcorineqtwo}
\forall\,E&\in\rgadom\cap\epsmo\rot\rganom&
\normltepsom{E}&\leq\cmgadrotepsid\normltom{\rot E},\\
\mylabel{estrotrotcmaxcorineqthree}
\forall\,H&\in\rganom\cap\epsmo\rot\rgadom&
\normltepsom{H}&\leq\cmganrotepsid\normltom{\rot H}
\end{align}
hold, where these sharp constants do not need to coincide if $\eps\neq\id$.
\end{cor}

\begin{lem}
\mylabel{estrotrotcmaxepsmuest}
It holds
\begin{itemize}
\item[\bf(i)]
$\epsu^{-2}\cmgadrot\leq\cmgadroteps\leq\epso^2\cmgadrot$,
\item[\bf(ii)]
$\cmgadrotepsid,\cmganrotepsid
\begin{cases}
\leq\min\{\epsu\cmgadroteps,\epso\cmgadrot\}
\leq\epso\cmgadrot,\\
\geq\max\{\epso^{-1}\cmgadroteps,\epsu^{-1}\cmgadrot\}
\geq\epsu^{-1}\cmgadrot.
\end{cases}$
\end{itemize}
\end{lem}

\begin{proof}
It is clear that $\cmgadrotepsid,\cmganrotepsid\leq\epsu\cmgadroteps$ holds.
To prove the other estimates, let $E\in\rgadom\cap\epsmo\rot\rganom$.
We decompose (see Appendix \ref{appmaxhelm}) 
$$E=E_{0}+E_{\rot}\in\rgadzom\oplus\rot\rganom.$$
Then $E_{\rot}\in\rgadom\cap\rot\rganom$ and $\rot E=\rot E_{\rot}$. 
Thus by orthogonality
$$\normltepsom{E}^2
=\scpLtom{\eps E}{E_{\rot}}
\leq\cmgadrot
\underbrace{\normltom{\eps E}}_{\leq\epso\normltepsom{E}}
\normltom{\rot E}$$
and hence
$$\normltepsom{E}
\leq\epso\cmgadrot\normltom{\rot E}
\leq\epso^2\cmgadrot\normltepsmoom{\rot E}.$$
This shows $\cmgadrotepsid\leq\epso\cmgadrot$
and $\cmgadroteps\leq\epso^2\cmgadrot$.
Interchanging $\gad$ and $\gan$ proves 
$\cmganrotepsid\leq\epso\cmganrot=\epso\cmgadrot$.
By $\epsu^{-1}\normltom{E}\leq\normltepsom{E}$ 
and \eqref{estrotrotcmaxcorineqone} 
resp.~\eqref{estrotrotcmaxcorineqtwo}
resp.~\eqref{estrotrotcmaxcorineqthree}
we see $\cmgadrot\leq\epsu^2\cmgadroteps$ 
resp.~$\epsu^{-1}\cmgadrot\leq\cmgadrotepsid,\cmganrotepsid$.
Using $\normltom{\rot E}\leq\epso\normltepsmoom{\rot E}$ 
and \eqref{estrotrotcmaxcorineqtwo}, \eqref{estrotrotcmaxcorineqthree}
we get $\epso^{-1}\cmgadroteps\leq\cmgadrotepsid,\cmganrotepsid$,
which completes the proof.
\end{proof}

\subsubsection{The Full Maxwell Estimates}

\begin{theo}
\mylabel{maintheononconvex}
For all $E\in\rgadom\cap\epsmo\dganom$
the tangential-normal Maxwell estimate
$$\normltepsom{E-\pidn E}^2
\leq\cpgadeps^2\normltom{\div\eps E}^2+\cmgadrotepsid^2\normltom{\rot E}^2$$
holds with sharp constants.
Moreover, $\cpgadeps\leq\epsu\cpgad$ 
and $\cmgadrotepsid\leq\epso\cmgadrot$.
\end{theo}

\begin{proof}
By the Helmholtz decomposition (see Appendix \ref{appmaxhelm}) we have
$$\rgadom\cap\epsmo\dganom\cap\harmdnepsom^{\bot_{\eps}}\ni
E-\pidn E=E_{\na}+E_{\rot}\in\na\hogadom\oplus_{\eps}\epsmo\rot\rganom$$
with
\begin{align*}
E_{\na}
&\in\epsmo\dganom\cap\na\hogadom
=\rgadzom\cap\epsmo\dganom\cap\harmdnepsom^{\bot_{\eps}},&
\div\eps E_{\na}
&=\div\eps E,\\
E_{\rot}
&\in\rgadom\cap\epsmo\rot\rganom
=\rgadom\cap\epsmo\dganzom\cap\harmdnepsom^{\bot_{\eps}},&
\rot E_{\rot}
&=\rot E.
\end{align*}
Thus, by Lemma \ref{estgraddivcpdiv} and Corollary \ref{estrotrotcmaxcor}
as well as orthogonality we obtain
$$\normltepsom{E-\pidn E}^2
=\normltepsom{E_{\na}}^2+\normltepsom{E_{\rot}}^2
\leq\cpgadeps^2\normltom{\div\eps E}^2
+\cmgadrotepsid^2\normltom{\rot E}^2.$$
Lemmas \ref{estgraddivcpdivepsest} and \ref{estrotrotcmaxepsmuest}
show the two estimates for the constants,
completing the proof.
\end{proof}

\begin{lem}
\mylabel{maintheononconvexlem}
It holds
$$\cmgadeps
=\max\{\cpgadeps,\cmgadrotepsid\}
\begin{cases}
\leq\max\{\epsu\cpgad,\epso\cmgadrot\}
\leq\epsh\max\{\cpgad,\cmgadrot\}\\
\geq\max\{\epso^{-1}\cpgad,\epsu^{-1}\cmgadrot\}
\geq\epsch\max\{\cpgad,\cmgadrot\}
\end{cases}$$
and for $\eps=\id$
$$\cmgad=\max\{\cpgad,\cmgadrot\}.$$
\end{lem}

\begin{proof}
We have $\cmgadeps\leq\max\{\cpgadeps,\cmgadrotepsid\}$.
Inserting $E\in\epsmo\dganom\cap\na\hogadom$ resp.
$E\in\rgadom\cap\epsmo\rot\rganom$ into the tangential-normal Maxwell estimate
\eqref{intro-estrotdiv} shows $\cpgadeps,\cmgadrotepsid\leq\cmgadeps$ 
and the first equation follows.
The other estimates are given by Lemmas \ref{estgraddivcpdivepsest} and \ref{estrotrotcmaxepsmuest},
completing the proof.
\end{proof}

By the latter theorem and lemma it remains to estimate only 
the two constants $\cpgad$ and $\cmgadrot$ for the various $\gad$.

\subsection{Full Boundary Conditions}

We summarize our results for the two important extreme cases $\gad=\ga$ resp.~$\gad=\emptyset$, i.e.,
the full tangential resp.~the full normal case, and emphasize that in these two cases
the tangential and normal Maxwell rotation constants coincide by \eqref{rotconsteq} 
and hence beside the Poincar\'e constants we just have to estimate one constant, namely
\begin{align}
\mylabel{cmrotgaeqempty}
\cmroteps:=\cmgaroteps=\cmemptyroteps,\quad
\cmrot=\cmgarot=\cmemptyrot.
\end{align}

For the convenience of the reader
let us recall our estimates from the latter sections in these two extreme cases.
Lemmas \ref{estgraddivcpdiv} and \ref{estgraddivcpdivepsest} read:

\begin{cor}
\mylabel{estgraddivcpdivfullboundary}
The Poincar\'e constant $\cpgaeps$ in $\hogaom$ 
resp.~$\cpeps$ in $\hoemptyom$
and the Maxwell divergence constant $\cmemptydiveps$ in $\epsmo\dom\cap\na\hogaom$ 
resp.~$\cmgadiveps$ in $\epsmo\dgaom\cap\na\hoom$ equal, i.e.,
the inequalities
\begin{align*}
\forall\,u&\in\hogaom&
\normltom{u}&\leq\cpgaeps\normltepsom{\na u},\\
\forall\,E&\in\epsmo\dom\cap\na\hogaom&
\normltepsom{E}&\leq\cpgaeps\normltom{\div\eps E}
\intertext{resp.}
\forall\,u&\in\hoom\cap\rz^{\bot}&
\normltom{u}&\leq\cpeps\normltepsom{\na u},\\
\forall\,E&\in\epsmo\dgaom\cap\na\hoom&
\normltepsom{E}&\leq\cpeps\normltom{\div\eps E}
\end{align*}
hold with sharp constants. 
Moreover, $\epso^{-1}\cpga\leq\cpgaeps\leq\epsu\cpga$
and $\epso^{-1}\cp\leq\cpeps\leq\epsu\cp$.
\end{cor}

Here, $\cpeps:=\cpemptyeps$. 
Corollary \ref{estrotrotcmaxcor} and Lemma \ref{estrotrotcmaxepsmuest} read:

\begin{cor}
\mylabel{estrotrotcmaxcorfullboundary}
The tangential Maxwell rotation constant $\cmgaroteps$
in $\rgaom\cap\epsmo\rot\rom$ 
and the normal Maxwell rotation constant $\cmemptyroteps$
in $\rom\cap\epsmo\rot\rgaom$ equal, i.e.,
for all $E\in\big(\rgaom\cap\epsmo\rot\rom\big)\cup\big(\rom\cap\epsmo\rot\rgaom\big)$
$$\normltepsom{E}
\leq\cmroteps\normltepsmoom{\rot E}
\leq\epsu\cmroteps\normltom{\rot E}$$
holds with sharp constants. Moreover, the inequalities 
\begin{align*}
\forall\,E&\in\rgaom\cap\epsmo\rot\rom&
\normltepsom{E}&\leq\cmgarotepsid\normltom{\rot E},\\
\forall\,H&\in\rom\cap\epsmo\rot\rgaom&
\normltepsom{H}&\leq\cmemptyrotepsid\normltom{\rot H}
\end{align*}
hold, where these sharp constants do not need to coincide if $\eps\neq\id$.
Moreover, it holds $\epsu^{-2}\cmrot\leq\cmroteps\leq\epso^2\cmrot$ and
\begin{align*}
\epsu^{-1}\cmrot
\leq\max\{\epso^{-1}\cmroteps,\epsu^{-1}\cmrot\}
&\leq\cmgarotepsid,\cmemptyrotepsid\\
&\leq\min\{\epsu\cmroteps,\epso\cmrot\}
\leq\epso\cmrot.
\end{align*}
\end{cor}

Theorem \ref{maintheononconvex} and Lemma \ref{maintheononconvexlem} read:

\begin{cor}
\mylabel{maintheononconvexfullboundary}
For all $E\in\rgaom\cap\epsmo\dom$ and all $H\in\rom\cap\epsmo\dgaom$
the tangential and normal Maxwell estimates
\begin{align*}
\normltepsom{E-\pid E}^2
&\leq\cpgaeps^2\normltom{\div\eps E}^2+\cmgarotepsid^2\normltom{\rot E}^2,\\
\normltepsom{H-\pin H}^2
&\leq\cpeps^2\normltom{\div\eps H}^2+\cmemptyrotepsid^2\normltom{\rot H}^2
\end{align*}
hold with sharp constants.
Furthermore, the estimates $\epso^{-1}\cpga\leq\cpgaeps,\cpeps\leq\epsu\cp$ 
and $\epsu^{-1}\cmrot\leq\cmgarotepsid,\cmemptyrotepsid\leq\epso\cmrot$ as well as
\begin{align*}
\cmgaeps
&=\max\{\cpgaeps,\cmgarotepsid\}
\begin{cases}
\leq\max\{\epsu\cpga,\epso\cmrot\}
\leq\epsh\max\{\cpga,\cmrot\},\\
\geq\max\{\epso^{-1}\cpga,\epsu^{-1}\cmrot\}
\geq\epsch\max\{\cpga,\cmrot\},
\end{cases}\\
\cmemptyeps
&=\max\{\cpeps,\cmemptyrotepsid\}
\begin{cases}
\leq\max\{\epsu\cp,\epso\cmrot\}
\leq\epsh\max\{\cp,\cmrot\},\\
\geq\max\{\epso^{-1}\cp,\epsu^{-1}\cmrot\}
\geq\epsch\max\{\cp,\cmrot\}
\end{cases}
\end{align*}
hold. Therefore, in both cases
\begin{align*}
\epsch\max\{\cpga,\cmrot\}
\leq\max\{\epso^{-1}\cpga,\epsu^{-1}\cmrot\}
&\leq\cmgaeps,\cmemptyeps\\
&\leq\max\{\epsu\cp,\epso\cmrot\}
\leq\epsh\max\{\cp,\cmrot\}.
\end{align*}
For $\eps=\id$ it holds
$$\cmga=\max\{\cpga,\cmrot\},\quad
\cmempty=\max\{\cp,\cmrot\}.$$
\end{cor}

As the two Poincar\'e constants $\cpga<\cp$ are more or less well known,
by the latter corollaries it remains only to estimate the Maxwell constant $\cmrot$.

\subsubsection{Convex Domains}

Now, let $\om\subset\rt$ be a bounded and \ul{convex} domain.
Then $\om$ is strongly Lipschitz, see e.g.~\cite[Corollary 1.2.2.3]{grisvardbook}.
Moreover, there are no Dirichlet or Neumann fields since $\om$ is simply connected 
and has a connected boundary.
As noted before in \eqref{friedrichspoincareconstest},
in the convex case we can estimate the Poincar\'e constant $\cp$ 
by the diameter of $\om$, i.e.,
$$\cpga<\cp\leq\frac{\diam(\om)}{\pi}.$$
We show that we can also estimate the Maxwell constant $\cmrot$
in the two extreme cases $\gad=\ga$ resp.~$\gad=\emptyset$ by $\cp$.
In \cite[Theorem 2.17]{amrouchebernardidaugegiraultvectorpot} 
the following crucial lemma has been proved,
which is the key point in our investigations for convex domains.

\begin{lem}
\mylabel{french}
Let $E$ belong to $\rgaom\cap\dom$ or $\rom\cap\dgaom$.
Then $E\in\hoom$ and
\begin{align}
\mylabel{frenchformula}
\normltom{\na E}^2\leq\normltom{\rot E}^2+\normltom{\div E}^2.
\end{align}
\end{lem}

We note that the latter lemma has already been
proved in \cite{saranenineqfried} in the case $\rgaom\cap\dom$.

\begin{rem}
\mylabel{frenchrem}
For $E\in\hogaom$ it is clear that for any domain $\om\subset\rt$
(or even in $\rN$)
$$\normltom{\na E}^2=\normltom{\rot E}^2+\normltom{\div E}^2$$
holds since $-\Delta=\rot\rot-\na\div$.
In general, this formula is no longer valid if $E$ has just the tangential
or normal boundary condition. 
\end{rem}

With the help of Lemma \ref{french} we can now estimate $\cmrot$.

\begin{lem}
\mylabel{mainlem}
$\cmrot\leq\cp$. More precisely, for all $E$ 
in $\rgaom\cap\rot\rom$ or $\rom\cap\rot\rgaom$
$$\normltom{E}\leq\cp\normltom{\rot E}.$$
Furthermore, $\cmgarotepsid,\cmemptyrotepsid\leq\epso\cp$.
\end{lem}

\begin{proof}
By \eqref{cmrotgaeqempty} the boundary condition does not matter.
So, let 
$$E\in\rom\cap\rot\rgaom=\rom\cap\dgazom$$ 
with $E=\rot H$ for some $H\in\rgaom$. 
Then, for any constant vector $a\in\rt$
\begin{align}
\mylabel{orthort}
\scpLtom{E}{a}=\scpLtom{\rot H}{a}=0
\end{align}
holds. Thus, by Poincar\'e's estimate and Lemma \ref{french} we get $E\in\hoom\cap(\rt)^{\bot}$ and
$$\normltom{E}\leq\cp\normltom{\na E}\leq\cp\normltom{\rot E},$$
which shows $\cmrot=\cmemptyrot\leq\cp$.
\end{proof}

We can now formulate the main result for convex domains, which follows immediately
from Corollary \ref{maintheononconvexfullboundary} and Lemma \ref{mainlem}.

\begin{theo}
\mylabel{maintheoconvexfullboundary}
For all $E\in\rgaom\cap\epsmo\dom$ and all $H\in\rom\cap\epsmo\dgaom$
the tangential and normal Maxwell estimates
\begin{align*}
\normltepsom{E}^2
&\leq\epsu^2\cpga^2\normltom{\div\eps E}^2+\epso^2\cp^2\normltom{\rot E}^2,\\
\normltepsom{H}^2
&\leq\epsu^2\cp^2\normltom{\div\eps H}^2+\epso^2\cp^2\normltom{\rot H}^2
\end{align*}
hold. Moreover, 
$$\frac{\cpga}{\epso}
\leq\cmgaeps
\leq\epsh\cp,\quad
\frac{\cp}{\epso}
\leq\cmemptyeps
\leq\epsh\cp.$$
Especially, for $\eps=\id$
$$\max\{\cpga,\cmrot\}=\cmga\leq\cmempty=\cp.$$
\end{theo}

\begin{theo}
\mylabel{maintheocor}
For all $E\in\big(\rgaom\cap\epsmo\dom\big)\cup\big(\rom\cap\epsmo\dgaom\big)$
$$\normltepsom{E}
\leq\epsh\cp\big(\normltom{\div\eps E}^2+\normltom{\rot E}^2\big)^{\nicefrac{1}{2}}.$$
\end{theo}

\begin{acknow}
The author is deeply indebted to Sergey Repin for bringing his attention
to the problem of the Maxwell constants in 3D
and to Sebastian Bauer und Karl-Josef Witsch for so many fruitful and nice discussions.
\end{acknow}

\bibliographystyle{plain} 
\bibliography{/Users/paule/Library/texmf/tex/TeXinput/bibtex/paule}

\appendix
\section{Appendix}

\subsection{More General Operators}
\mylabel{appgenop}

There are obvious generalizations to differential forms.
Let $\om$ be a smooth Riemannian manifold 
of dimension $N\geq2$
with boundary $\ga$ and compact closure.
We assume that the boundary manifold $\ga$ is divided into two
$(N-1)$-dimensional Riemannian sub-manifolds $\gad$ and $\gan$ with boundaries. 
Let us denote by $\ltqom$ the usual Lebesgue (Hilbert) space of $q$-forms.
For the exterior derivative and co-derivative we define the well known Sobolev spaces
$$\dqom:=\set{E\in\ltqom}{\ed E\in\ltqpoom},\quad
\deqom:=\set{E\in\ltqom}{\cd E\in\ltqmoom}.$$
As before, we introduce weak homogeneous boundary conditions 
by closures of respective test forms, yielding the Sobolev spaces
$$\dqgadom,\quad
\deqganom.$$
Let $\A$ be
$$\mumo\ed:\dqgadom\subset\ltqepsom\to\ltqpomuom.$$
Then $\As$ is
$$-\epsmo\cd:\deqpoganom\subset\ltqpomuom\to\ltqepsom,$$
where $\eps$ resp.~$\mu$ are bounded, symmetric, real and uniformly positive definite
linear transformations on $q$- resp.~$(q+1)$-forms.
More precisely:
\begin{center}\begin{tabular}{|c||c|c|c||c|c|}
\hline
$\A$ & $D(\A)$ & $\X$ & $\Y$ & $N(\A)$ & $R(\A)$ \\
\hline
$\mumo\ed$ & $\dqgadom$ & $\ltqepsom$ & $\ltqpomuom$ & $\dqgadzom$ & $\mumo\ed\dqgadom$\\
\hline\hline
$\As$ & $D(\As)$ & $\Y$ & $\X$ & $N(\As)$ & $R(\As)$\\
\hline
$-\epsmo\cd$ & $\deqpoganom$ & $\ltqpomuom$ & $\ltqepsom$ & $\deqpoganzom$ & $\epsmo\cd\deqpoganom$\\
\hline
\end{tabular}\end{center}
Here, 
$$\dqgadzom:=\set{E\in\dqgadom}{\ed E=0},\quad
\deqganzom:=\set{E\in\deqganom}{\cd E=0}$$
and we note
$$R(\A)=\mumo\big(\dqpogadzom\cap\harmqpodnom^{\bot}\big),\quad
R(\As)=\epsmo\big(\deqganzom\cap\harmqdnom^{\bot}\big),$$
where $\harmqdnom:=\dqgadzom\cap\deqganzom$.
Indeed $D(\As)=\deqpoganom$ holds.
We have the same remarks as in Section \ref{secrotations}.
Again, for this one has to show the approximation property
$$\deqpoganom=\set{H\in\deqpoom}{\scpltqom{\cd H}{E}=-\scpltqpoom{H}{\ed E}\,\forall\;E\in\dqgadom},$$
which is not trivial at all for mixed boundary conditions.
And again, only in the special cases of full boundary conditions this is clear.
Since $D(\As)=\deqpoom$ holds for $\gad=\ga$ by definition
we have also $D(\Bs)=D(\A^{**})=D(\A)=\dqgaom$ for $\B=\As$,
which shows the result for $\gad=\emptyset$.
The crucial compact embedding \eqref{compactembA} is
$$\dqgadom\cap\epsmo\ol{\cd\deqpoganom}\hookrightarrow\ltqepsom.$$
Both latter properties of $\om$, i.e., the approximation and the compactness property,
hold, e.g., if the boundary manifolds $\ga$, $\gad$, $\gan$ are Lipschitz
and the boundary manifolds $\gad$, $\gan$ are separated
by a $(N-2)$-dimensional Riemannian and Lipschitz sub-manifold,
the interface $\gamma:=\ol{\gad}\cap\ol{\gan}$,
see \cite{goldshteinmitreairinamariushodgedecomixedbc,jakabmitreairinamariusfinensolhodgedeco}
for details and proofs. We note that
$$\dqgadom\cap\epsmo\ol{\cd\deqpoganom}
\subset\dqgadom\cap\epsmo\deqganzom
\subset\dqgadom\cap\epsmo\deqganom$$
holds and that even the compact embedding of the latter space into $\ltqom$, this is
$$\dqgadom\cap\epsmo\deqganom\hookrightarrow\ltqom\subset\ltqepsom,$$
has been shown in \cite{jakabmitreairinamariusfinensolhodgedeco}\footnote{In
\cite{jakabmitreairinamariusfinensolhodgedeco} it is proved that
$\dqgadom\cap\deqganom$ even embeds continuously to $\Hgen{1/2,q}{}{}(\om)$
and hence compactly to $\ltqom$.
We note that the compactness property is independent of $\eps$,
see e.g.~\cite{kuhndiss}.}.
By Theorem \ref{fatheothree} we have
$$\kappa:=\min_{0\neq E\in\dqgadom\cap\epsmo\cd\deqpoganom}
\frac{\norm{\mumo\ed E}_{\ltqpomuom}}{\norm{E}_{\ltqepsom}}
=\min_{0\neq H\in\deqpoganom\cap\mumo\ed\dqgadom}
\frac{\norm{\epsmo\cd H}_{\ltqepsom}}{\norm{H}_{\ltqpomuom}}$$
and $\kappa^2$ is the first positive 
Dirichlet-Neumann eigenvalue of the weighted reduced $\cd$-$\ed$-operator $-\epsmo\cd\mumo\ed$.
Analogously $\kappa^2$ is also the first positive
Neumann-Dirichlet eigenvalue of the weighted reduced $\ed$-$\cd$-operator $-\mumo\ed\epsmo\cd$.

\begin{lem}
\mylabel{estedcdcmax}
The tangential-normal and normal-tangential generalized Maxwell constants, i.e.,
the best constants in the inequalities
\begin{align*}
\forall\,E&\in\dqgadom\cap\epsmo\cd\deqpoganom&
\norm{E}_{\ltqepsom}&\leq c_{\mathtt{gm},\gad,\ed,\eps,\mu}\norm{\ed E}_{\ltqpomumoom},\\
\forall\,H&\in\deqpoganom\cap\mumo\ed\dqgadom&
\norm{H}_{\ltqpomuom}&\leq c_{\mathtt{gm},\gan,\cd,\mu,\eps}\norm{\cd H}_{\ltqepsmoom},
\end{align*}
coincide and equal to $1/\kappa$, i.e., 
$c_{\mathtt{gm},\gad,\ed,\eps,\mu}=c_{\mathtt{gm},\gan,\cd,\mu,\eps}=\kappa^{-1}$.
\end{lem}

\begin{rem}
\mylabel{remdiffformsapp}
It is clear that more results of this contribution
can be generalized to the differential form setting.
\end{rem}

\subsection{Maxwell Tools}

Let the general assumptions from the introduction be satisfied.

\subsubsection{The Maxwell Estimates}
\mylabel{appmaxest}

By the Maxwell compactness property we get immediately the Maxwell estimate.

\begin{lem}
\mylabel{lemmaxest}
There exists $\cmgadeps>0$, such that
for all $E$ in $\rgadom\cap\epsmo\dganom\cap\harmdnepsom^{\bot_{\eps}}$ 
$$\normltepsom{E}
\leq\cmgadeps\big(\normltom{\rot E}^2+\normltom{\div\eps E}^2\big)^{\nicefrac{1}{2}}.$$
\end{lem}

\begin{proof}
If the estimate would not hold, there would exist a sequence of vector fields
$(E_{n})\subset\rgadom\cap\epsmo\dganom\cap\harmdnepsom^{\bot_{\eps}}$
with $\normltepsom{E_{n}}=1$ and
$$\normltom{\rot E_{n}}+\normltom{\div\eps E_{n}}<\frac{1}{n}.$$
By the Maxwell compactness property we can assume w.l.o.g. that
$(E_{n})$ converges in $\ltepsom$ to some $E\in\ltepsom$.
By testing, $E$ belongs to $\rzom\cap\epsmo\dzom\cap\harmdnepsom^{\bot_{\eps}}$
and $(E_{n})$ converges to $E$ also in $\rom\cap\epsmo\dom$.
As $\rgadom$ resp.~$\dganom$ is a closed subspace of $\rom$ resp.~$\dom$, 
$E$ belongs even to $\rgadzom\cap\epsmo\dganzom=\harmdnepsom$.
Hence, $E=0$, which contradicts $1=\normltom{E_{n}}\to0$.
\end{proof}

\begin{cor}
\mylabel{lemmaxestcor}
For all $E$ in $\rgadom\cap\epsmo\dganom$ 
$$\normltepsom{(1-\pidn)E}
\leq\cmgadeps\big(\normltom{\rot E}^2+\normltom{\div\eps E}^2\big)^{\nicefrac{1}{2}}.$$
\end{cor}

\begin{proof}
As $H:=(1-\pidn)E\in\rgadom\cap\epsmo\dganom\cap\harmdnepsom^{\bot_{\eps}}$
with $\div\eps H=\div\eps E$ and $\rot H=\rot E$, 
Lemma \ref{lemmaxest} completes the proof.
\end{proof}

The same arguments show that the Maxwell estimate remains valid in any dimension 
and even for compact Riemannian manifolds, as long as the crucial Maxwell compactness property holds.

\subsubsection{Helmholtz-Weyl Decompositions}
\mylabel{appmaxhelm}

By the projection theorem we have for the operator $\na$
$$\ltepsom=\ol{\na\hogadom}\oplus_{\eps}\epsmo\dganzom,$$
where indeed $\big(\na\hogadom\big)^{\bot}=\dganzom$
holds by \cite{jochmanncompembmaxmixbc}.
Note that $\na\hogadom$ is already closed by Rellich's selection theorem.
Analogously, we obtain for the operator $\rot$
\begin{align}
\mylabel{helmrot}
\ltepsom=\rgadzom\oplus_{\eps}\epsmo\ol{\rot\rganom},
\end{align}
where again and indeed $\big(\rot\rganom\big)^{\bot}=\rgadzom$
holds by \cite{jochmanncompembmaxmixbc}.
For $\eps=\id$ we get by \eqref{helmrot}
$$\rgadom=\rgadzom\oplus\big(\rgadom\cap\ol{\rot\rganom}\big)$$
and therefore
$$\rot\rgadom=\rot\big(\rgadom\cap\ol{\rot\rganom}\big).$$
As $\ol{\rot\rganom}\subset\dganzom\cap\harmdnom^{\bot}$,
the Maxwell estimate Lemma \ref{lemmaxest}
implies that also $\rot\rgadom$ is already closed. Moreover,
$$\rot\rgadom=\rot\rcalgadom,\quad
\rcalgadom:=\rgadom\cap\rot\rganom=\rgadom\cap\rot\rcalganom.$$
Since $\na\hogadom\subset\rgadzom$ and $\rot\rganom\subset\dganzom$ we obtain
\begin{align*}
\rgadzom
&=\na\hogadom\oplus_{\eps}\big(\underbrace{\rgadzom\cap\epsmo\dganzom}_{\hspace*{5em}=\harmdnepsom}\big),\\
\epsmo\dganzom
&=\epsmo\rot\rganom\oplus_{\eps}\big(\overbrace{\rgadzom\cap\epsmo\dganzom}\big).
\end{align*}
Finally, we have the well known Helmholtz decompositions:

\begin{lem}
\mylabel{hemldeco}
It holds
\begin{align*}
\ltepsom
&=\na\hogadom\oplus_{\eps}\epsmo\dganzom=\rgadzom\oplus_{\eps}\epsmo\rot\rcalganom\\
&=\na\hogadom\oplus_{\eps}\harmdnepsom\oplus_{\eps}\epsmo\rot\rcalganom
\end{align*}
as well as
$$\na\hogadom=\rgadzom\cap\harmdnepsom^{\bot_{\eps}},\quad
\epsmo\rot\rcalganom=\epsmo\dganzom\cap\harmdnepsom^{\bot_{\eps}}$$
and $\rcalgadom=\rgadom\cap\dganzom\cap\harmdnom^{\bot}$.
\end{lem}

\subsection{Functional Analytical Tools}
\mylabel{appfa}

Let us recall that for a self-adjoint operator 
$\T:D(\T)\subset\H\to\H$, where $\H$ denotes some Hilbert space,
$$\cn\setminus\rz\subset\resol(\T),\quad
\spec(\T)=\specp(\T)\cup\specc(\T)\subset\rz,\quad
\specr(\T)=\emptyset$$
hold. Here, $\resol(\T)$, $\spec(\T)$, $\specp(\T)$, $\specc(\T)$, $\specr(\T)$
denote the resolvent set, the spectrum, the point spectrum, 
the continuous spectrum and the residual spectrum, respectively.
Moreover, we have the `Helmholtz' decompositions
$$\H=N(\T-\bar{\lambda})\oplus\ol{R(\T-\lambda)}.$$
For $\lambda\in\resol(\T)$ the continuity of $(\T-\lambda)^{-1}$ is equivalent to
$$\exists\,c>0\quad
\forall\,u\in D(\T)\quad
\norm{u}_{\H}\leq c\norm{(\T-\lambda)u}_{\H}.$$
Hence, as $\T$ is closed, $R(\T-\lambda)=\H$ holds for $\lambda\in\resol(\T)$,
see e.g.~\cite[VIII.1, Theorem]{yosidabook}.
Thus the resolvent set $\resol(\T)$, i.e.,
the set of all $\lambda\in\cn$ with
$N(\T-\lambda)=\{0\}$, $\ol{R(\T-\lambda)}=\H$ and
$(\T-\lambda)^{-1}:R(\T-\lambda)\to D(\T-\lambda)$ bounded,
is just given by
$$\resol(\T)
=\set{\lambda\in\cn}{(\T-\lambda)^{-1}:\H\to D(\T)\text{ bounded}}.$$
We note that for all $\lambda\in\cn$ the norms in 
$D(\T-\lambda)$ and $D(\T)$ are equivalent.

%


We give simple proofs of the results of section \ref{fa}.
For this, we recall the Hilbert spaces $\X$ and $\Y$ and 
the closed and densely defined linear operator $\A:D(\A)\subset\X\mapsto\Y$
with adjoint $\As:D(\As)\subset\Y\mapsto\X$.
$\As\A:D(\As\A)\subset\X\mapsto\X$ and $\A\As:D(\A\As)\subset\Y\mapsto\Y$
are self-adjoint and non-negative.
Furthermore, we introduce the Maxwell-type operator
$$\M:D(\M)\subset\Z\to\Z,\quad
D(\M):=D(\A)\times D(\As),\quad
\Z:=\X\times\Y$$
by $\M(x,y)=(\As y,\A x)$ and note that
$$\M=\zmat{0}{\As}{\A}{0},\quad
\M^2=\zmat{\As\A}{0}{0}{\A\As}$$
are self-adjoint as well and $\M^2$ is non-negative.
Moreover, we introduce two projections 
$\pi_{\X}:\Z\to\X$ and $\pi_{\Y}:\Z\to\Y$
by $\pi_{\X}z:=x$ and $\pi_{\Y}z:=y$ for $z=(x,y)$
and two embeddings
$\iota_{\X}:\X\to\Z$ and $\iota_{\Y}:\Y\to\Z$
by $\iota_{\X}x:=(x,0)$ and $\iota_{\Y}y:=(0,y)$.

First we show a stronger version of \eqref{spectrumAAs}.

\begin{lem}
\mylabel{spectralem}
It holds
\begin{itemize}
\item[\bf(i)] 
$0\in\spec(\M)\equi0\in\spec(\As\A)\cup\spec(\A\As)$,
\item[\bf(i')] 
$0\in\specc(\M)\equi0\in\specc(\As\A)\equi0\in\specc(\A\As)$,
\item[\bf(i'')] 
$0\in\specp(\M)\equi0\in\specp(\As\A)\cup\specp(\A\As)$
\end{itemize}
and for $\lambda\in\rz\setminus\{0\}$
\begin{itemize}
\item[\bf(ii)] 
$\lambda\in\resol(\M)\equi\lambda^2\in\resol(\As\A)\equi\lambda^2\in\resol(\A\As)$,
\item[\bf(ii')] 
$\lambda\in\spec(\M)\equi\lambda^2\in\spec(\As\A)\equi\lambda^2\in\spec(\A\As)$,
\item[\bf(iii)] 
$\lambda\in\specc(\M)\equi\lambda^2\in\specc(\As\A)\equi\lambda^2\in\specc(\A\As)$,
\item[\bf(iv)] 
$\lambda\in\specp(\M)\equi\lambda^2\in\specp(\As\A)\equi\lambda^2\in\specp(\A\As)$.
More precisely:
If $z:=(x,y)$ is an eigenvector to the eigenvalue $\lambda$ of $\M$, then
$x$ is an eigenvector to the eigenvalue $\lambda^2$ of $\As\A$ 
and $y$ is an eigenvector to the eigenvalue $\lambda^2$ of $\A\As$.
If $x$ is an eigenvector to the eigenvalue $\lambda^2$ of $\As\A$, then 
$z_{\pm}:=(x,\pm\lambda^{-1}\A x)$ is an eigenvector to the eigenvalue $\pm\lambda$ of $\M$, respectively.
If $y$ is an eigenvector to the eigenvalue $\lambda^2$ of $\A\As$, then 
$z_{\pm}:=(\pm\lambda^{-1}\As y,y)$ is an eigenvector to the eigenvalue $\pm\lambda$ of $\M$, respectively.
\end{itemize}
Therefore,
\begin{itemize}
\item[\bf(v)] 
$\resol(\M)$ and $\spec(\M)$, $\specc(\M)$, $\specp(\M)$ are point symmetric to the origin.
\end{itemize}
\end{lem}

\begin{proof}
As (i')$\wedge$(i'')$\impl$(i), (ii)$\impl$(ii') and (ii')$\wedge$(iv)$\impl$(iii), 
we only have to show (i'), (i''), (ii) and (iv). Then (v) is clear.

\ul{(ii):} 
We just show the assertions for $\As\A$. 
The corresponding results for $\A\As$ can be proven analogously.

$\Rightarrow:$ 
Let $\lambda\in\resol(\M)$, i.e.,
$N(\M-\lambda)=\{0\}$ and
$(\M-\lambda)^{-1}:\Z\to D(\M)$ is continuous. 

$\bullet$ 
First we show $N(\As\A-\lambda^2)=\{0\}$.
Let $x\in N(\As\A-\lambda^2)$.
Then $z:=(x,y)\in D(\M)$ with $y:=\lambda^{-1}\A x\in D(\As)$
belongs to $N(\M-\lambda)$ since
$(\M-\lambda)z=(\As y-\lambda x,\A x-\lambda y)=0$.
Hence $z=0$, especially $x=0$.

$\bullet$ 
Let $f\in\X$. 
We want to solve $(\As\A-\lambda^2)x=f$ with $x\in D(\As\A)$.
Defining the `dual variable' $y:=\lambda^{-1}\A x\in D(\As)$
and $z:=(x,y)\in D(\M)$, the mixed formulation of this problem is 
$$\lambda(\As y-\lambda x)=f,\quad
\A x=\lambda y\quad
\equi\quad
(\M-\lambda)z=(\As y-\lambda x,\A x-\lambda y)=\lambda^{-1}(f,0).$$
These heuristic considerations suggest to set $x:=\pi_{\X}z\in D(\A)$
and $y:=\pi_{\Y}z\in D(\As)$ with 
$z:=\lambda^{-1}(\M-\lambda)^{-1}\iota_{\X}f\in D(\M)$. Then
$(\As y-\lambda x,\A x-\lambda y)
=(\M-\lambda)z
=\lambda^{-1}(f,0)$,
i.e., $\A x=\lambda y\in D(\As)$ and $(\As\A-\lambda^2)x=f$.
Moreover, $x$ depends continuously on $f$ since
$$\norm{x}_{D(\As\A)}
\leq\underbrace{\normx{x}}_{\leq\normz{z}}
+\underbrace{\normx{\As\A x}}_{=\normx{f}+\lambda^2\normx{x}}
\leq c\big(\normz{z}+\normx{f}\big)
\leq c\normx{f}.$$
Therefore, 
$(\As\A-\lambda^2)^{-1}
=\lambda^{-1}\pi_{\X}(\M-\lambda)^{-1}\iota_{\X}
:\X\to D(\As\A)$
is continuous and thus $\lambda^2\in\resol(\As\A)$. 


$\Leftarrow$: 
Let $\lambda^2\in\resol(\As\A)$, i.e., $N(\As\A-\lambda^2)=\{0\}$ and
$(\As\A-\lambda^2)^{-1}:\X\to D(\As\A)$
is continuous. 

$\bullet$ 
First we show $N(\M-\lambda)=\{0\}$. 
Let $z=(x,y)\in N(\M-\lambda)$. As 
$$(\As y-\lambda x,\A x-\lambda y)=(\M-\lambda)z=0,$$
$\A x=\lambda y\in D(\As)$ with $\As\A x=\lambda\As y=\lambda^2x$.
Hence, $x\in N(\As\A-\lambda^2)$ yields $x=0$ and $y=0$, i.e., $z=0$. 

$\bullet$ 
Let $h=(f,g)\in\Z$.
We want to solve $(\M-\lambda)z=h$ with $(x,y)=z\in D(\M)$.
As $(\As y-\lambda x,\A x-\lambda y)=(f,g)$,
$y\in D(\As)$ is already given by the second equation $\lambda y=\A x-g$,
if $x$ is known. Hence, rewriting everything in terms of $x$, this is
$$(f,g)
=\big(\lambda^{-1}\As(\A x-g)-\lambda x,\A x-(\A x-g)\big)
=\big(\lambda^{-1}\As(\A x-g)-\lambda x,g\big),$$
we see that we need to solve $\As(\A x-g)-\lambda^2x=\lambda f$.
Since $g$ does not belong to $D(\As)$ in general,
we cannot apply $(\As\A-\lambda^2)^{-1}$ directly.
The ansatz $x=\xt+\xh\in D(\A)$ with $\xh\in D(\As\A)$
leads to
\begin{align}
\mylabel{eqansatz}
\As(\A\xt-g)-\lambda^2\xt+(\As\A-\lambda^2)\xh
=\lambda f.
\end{align}
By the Lax-Milgram lemma we can solve, e.g., $\As(\A\xt-g)+\xt=\lambda f$.
More precisely, there exists a unique $\xt\in D(\A)$ with
\begin{align}
\mylabel{laxmilgram}
\forall\,\varphi\in D(\A)\quad
\scpy{\A \xt}{\A\varphi}+\scpx{\xt}{\varphi}
=\lambda\scpx{f}{\varphi}+\scpy{g}{\A\varphi}
\end{align}
depending continuously on $f$ and $g$ and hence on $h$, i.e., 
$\norm{\xt}_{{D(\A)}}\leq|\lambda|\normx{f}+\normy{g}\leq c\normz{h}$.
Let us denote this bounded linear operator mapping $h$ to $\xt$ 
by $\L:\Z\to D(\A)$. Now, \eqref{eqansatz} turns to
$$(\As\A-\lambda^2)\xh
=(1+\lambda^2)\xt.$$
The latter heuristic computations suggest to define $z:=(x,y)$ by
$$x:=\xt+(1+\lambda^2)(\As\A-\lambda^2)^{-1}\xt\in D(\A),\quad
y:=\lambda^{-1}(\A x-g)$$
with $\xt$ from \eqref{laxmilgram}.
$\xt\in D(\A)$ is uniquely defined and depends continuously on $h$, i.e.,
$\norm{\xt}_{{D(\A)}}\leq c\normz{h}$. Moreover, 
$\A\xt-g\in D(\As)$ and $\As(\A\xt-g)=\lambda f-\xt$
by \eqref{laxmilgram}. As $x-\xt\in D(\As\A)$, we get
$y=\lambda^{-1}\big(\A(x-\xt)+\A\xt-g\big)\in D(\As)$.
Thus, $z$ belongs to $D(\M)$. Since
$$\lambda\As y
=\As\A(x-\xt)+\As(\A\xt-g)
=(1+\lambda^2)\xt+\lambda^2(x-\xt)+\lambda f-\xt=\lambda^2x+\lambda f$$
we obtain
$$(\M-\lambda)z
=(\As y-\lambda x,\A x-\lambda y)
=(f,g)=h.$$
Furthermore, $z$ depends continuously on $h$, i.e., using
$$\forall\,\varphi\in D(\As\A)\quad
\normy{\A\varphi}^2
=\scpx{\As\A\varphi}{\varphi}
\leq\normx{\As\A\varphi}\normx{\varphi}
\leq\normx{\varphi}^2+\normx{\As\A\varphi}^2$$
we have
\begin{align*}
\norm{z}_{D(\M)}
&\leq\norm{x}_{D(\A)}+\norm{y}_{D(\As)}
\leq c\big(\norm{x}_{D(\A)}+\normx{f}+\normy{g}\big)
\leq c\big(\norm{x-\xt}_{D(\A)}+\norm{\xt}_{D(\A)}+\normz{h}\big)\\
&\leq c\big(\norm{x-\xt}_{D(\As\A)}+\norm{\xt}_{D(\A)}+\normz{h}\big)
\leq c\big(\norm{\xt}_{D(\A)}+\normz{h}\big)
\leq c\normz{h}.
\end{align*}
Therefore, with $\chi:D(\A)\to\Z$ defined by $\chi(x):=(x,\lambda^{-1}\A x)$
we finally obtain that
$$(\M-\lambda)^{-1}
=\chi\big(1+(1+\lambda^2)(\As\A-\lambda^2)^{-1}\big)\L-\lambda^{-1}\iota_{\Y}\pi_{\Y}
:\Z\to D(\M)$$
is bounded and hence $\lambda\in\resol(\M)$.

\ul{(iv):}
$\Rightarrow$: 
Let $\lambda\in\specp(\M)$ and $z:=(x,y)$ be an eigenvector 
to $\lambda$, i.e., $0\neq z\in N(\M-\lambda)$. As
$0=(\M-\lambda)z=(\As y-\lambda x,\A x-\lambda y)$,
neither $x$ nor $y$ can be zero. 
Moreover, since $\M z=\lambda z\in D(\M)$, $z\in N\big((\M+\lambda)(\M-\lambda)\big)$ holds, this is
$$0=(\M+\lambda)(\M-\lambda)z=(\M^2-\lambda^2)z
=\big((\As\A-\lambda^2)x,(\A\As-\lambda^2)y\big).$$
Thus, $0\neq x\in N(\As\A-\lambda^2)$ and $0\neq y\in N(\A\As-\lambda^2)$ 
yielding $\lambda^2\in\specp(\As\A)\cap\specp(\A\As)$.

$\Leftarrow$: 
Let $\lambda^2\in\specp(\As\A)$ and  $x$ be an eigenvector 
to $\lambda^2$, i.e.,
$0\neq x\in N(\As\A-\lambda^2)$.
Then $z_{\pm}:=(x,\pm\lambda^{-1}\A x)\in D(\M)$ and
$$(\M\mp\lambda)z_{\pm}
=(\pm\lambda^{-1}\As\A x\mp\lambda x,\A x-\lambda\lambda^{-1}\A x)
=\pm\lambda^{-1}(\As\A x-\lambda^2x,0)
=0.$$
Hence, $0\neq z_{\pm}\in N(\M\mp\lambda)$, i.e., $\pm\lambda\in\specp(\M)$.
Similar arguments apply to the case $\lambda^2\in\specp(\A\As)$.

\ul{(i'):} 
It holds with (ii')
\begin{align*}
0&\in\specc(\M)&
\equi&&
\exists\quad(\lambda_{n})&\subset\spec(\M)\setminus\{0\}&
\lambda_{n}\to0\\
&&
\equi&&
\exists\quad(\lambda_{n}^2)&\subset\spec(\As\A)\setminus\{0\}&
\lambda_{n}^2\to0\\
&&
\equi&&
0&\in\specc(\As\A)
\end{align*}
and the same is valid for $\A\As$.

\ul{(i''):} 
If $0\in\specp(\M)$, then there exists $0\neq z=(x,y)\in N(\M)$, i.e., $0=\M z=(\As y,\A x)$.
But then $0\neq z\in N(\M^2)$, i.e., $0=\M^2z=(\As\A x,\A\As y)$.
As either $x\neq0$ or $y\neq0$, we get $0\in\specp(\As\A)\cup\specp(\A\As)$.
Now, let e.g.~$0\in\specp(\As\A)$. Then, there exists 
$0\neq x\in N(\As\A)$, i.e., $\As\A x=0$. This implies $\A x=0$ since
$$0=\scpx{\As\A x}{x}=\scpy{\A x}{\A x}=\normy{\A x}^2.$$
Thus $0\neq z:=(x,0)\in N(\M)$ because $\M z=(\As 0,\A x)=0$.
Therefore, $0\in\specp(\M)$.
\end{proof}

We recall the `Helmholtz' decompositions
$$\X=N(\A)\oplus\ol{R(\As)},\quad
D(\A)=N(\A)\oplus\big(D(\A)\cap\ol{R(\As)}\big)$$
and define the restricted operator
$$\cA:=\A|_{D(\cA)}:D(\cA)\subset\ol{R(\As)}\to\ol{R(\A)},\quad
\cA x:=\A x,\quad
x\in D(\cA):=D(\A)\cap\ol{R(\As)}.$$
Let us compute the adjoint $\cAs:D(\cAs)\subset\ol{R(\A)}\to\ol{R(\As)}$.
For $y\in D(\cAs)$ we have for all $\varphi\in D(\cA)$
$$\scpy{\cA\varphi}{y}
=\scpx{\varphi}{\cAs y}.$$
Hence, for all $\psi=\psi_{0}+\varphi\in D(\A)=N(\A)\oplus D(\cA)$ 
we get with $\cA\varphi=\A\varphi=\A\psi$ 
and by $\cAs y\in\ol{R(\As)}\bot N(\A)$
$$\scpy{\A\psi}{y}
=\scpy{\cA\varphi}{y}
=\scpx{\varphi}{\cAs y}
=\scpx{\psi}{\cAs y}.$$
Thus, $y\in D(\As)$ and $\As y=\cAs y$.
This shows $D(\cAs)=D(\As)\cap\ol{R(\A)}$
and $\cAs:=\As|_{D(\cAs)}$, i.e.,
$$\cAs=\As|_{D(\cAs)}:D(\cAs)\subset\ol{R(\A)}\to\ol{R(\As)},\quad
\cAs y=\As y,\quad
y\in D(\cAs)=D(\As)\cap\ol{R(\A)}.$$
Moreover, we have $(\cAs)^{*}=\cA$ 
and the operators $\cAs\cA:D(\cAs\cA)\subset\ol{R(\As)}\to\ol{R(\As)}$ 
and $\cA\cAs:D(\cA\cAs)\subset\ol{R(\A)}\to\ol{R(\A)}$  
are self-adjoint and non-negative.
Finally, also the restriction 
$$\cM:=\M|_{D(\cM)}:D(\cM)\subset\ol{R(\M)}\to\ol{R(\M)},\quad
\cM z:=\M z,\quad
z\in D(\cM):=D(\M)\cap\ol{R(\M)}$$
is self-adjoint and we have
$$\cM=\zmat{0}{\cAs}{\cA}{0},\quad
\cM^2=\zmat{\cAs\cA}{0}{0}{\cA\cAs}.$$

\begin{rem}
\mylabel{helmholtzremapp}
Let us emphasize once more the `Helmholtz' decompositions
\begin{align*}
\X&=N(\A)\oplus\ol{R(\As)},&
D(\A)&=N(\A)\oplus D(\cA),\\
\Y&=N(\As)\oplus\ol{R(\A)},&
D(\As)&=N(\As)\oplus D(\cAs),\\
\Z&=N(\M)\oplus\ol{R(\M)},&
D(\M)&=N(\M)\oplus D(\cM).
\end{align*}
\end{rem}

We introduce the orthogonal projectors
$$\pi_{0}:\Z\to N(\M),\quad
\pi:\Z\to\ol{R(\M)}$$
and note $\pi|_{D(\M)}:D(\M)\to D(\cM)$.

\begin{lem}
\mylabel{lemAsAlemspec}
We have $0\not\in\specp(\cM)\cup\specp(\cAs\cA)\cup\specp(\cA\cAs)$. Moreover:
\begin{itemize}
\item[\bf(i)] 
The inverse operators 
$\cA^{-1}$, $(\cAs)^{-1}$ and $\cM^{-1}$ exist.
\item[\bf(ii)] 
$R(\A)=R(\cA),\quad 
R(\As)=R(\cAs),\quad 
R(\M)=R(\cM)$
\item[\bf(iii)] 
Lemma \ref{spectralem} holds for $\cA$, $\cAs$ and $\cM$ as well,
which follows immediately by
replacing $\X$ by $\ol{R(\As)}$ and $\Y$ by $\ol{R(\A)}$
as well as $\A$ by $\cA$ and $\As$ by $\cAs$.
\end{itemize}
\end{lem}

\begin{lem}
\mylabel{spectralemM}
It holds
\begin{itemize}
\item[\bf(i)] 
$\spec(\M)\setminus\{0\}=\spec(\cM)\setminus\{0\}$, 
more precisely even
$\specc(\M)\setminus\{0\}=\specc(\cM)\setminus\{0\}$ and
$\specp(\M)\setminus\{0\}=\specp(\cM)\setminus\{0\}$,
\item[\bf(ii)] 
$\resol(\M)\setminus\{0\}=\resol(\cM)\setminus\{0\}$,
\item[\bf(iii)] 
$\ds\specp(\cM^{-1})\setminus\{0\}=\frac{1}{\specp(\cM)\setminus\{0\}}$,
more precisely $N(\cM-\lambda)=N(\cM^{-1}-\lambda^{-1})$
for $\lambda\neq0$.
\end{itemize}
\end{lem}

\begin{proof}
We start with proving (ii).

$\Rightarrow$: Let $0\neq\lambda\in\resol(\M)$. 
We note that $R(\M-\lambda)=\Z$.
For $h\in\ol{R(\M)}\subset\Z$ we want solve $(\cM-\lambda)z=h$.
$z:=(\M-\lambda)^{-1}h\in D(\M)$ with $(\M-\lambda)z=h$
satisfies $\lambda z=\M z-h\in\ol{R(\M)}$ and thus $z\in D(\cM)$.
As $\norm{z}_{D(\cM)}=\norm{z}_{D(\M)}\leq c\normz{h}=c\norm{h}_{\ol{R(\M)}}$,
$z$ depends continuously on $h$. Hence $\lambda\in\resol(\cM)$.

$\Leftarrow$: Let $0\neq\lambda\in\resol(\cM)$. 
We note that $R(\cM-\lambda)=\ol{R(\M)}$.
For $h\in\Z$ we want solve $(\M-\lambda)z=h$.
Decomposing 
$$h=h_{0}+\tilde{h}\in\Z=N(\M)\oplus\ol{R(\M)},\quad
z=z_{0}+\zt\in D(\M)=N(\M)\oplus D(\cM)$$ 
shows with $\M\zt\in R(\M)$
$$-\lambda z_{0}+(\M-\lambda)\zt=h_{0}+\tilde{h}\quad\equi\quad
-\lambda z_{0}=h_{0}\,\wedge(\M-\lambda)\zt=\tilde{h}.$$
This gives rise to define $z\in D(\M)$ by
$$z:=z_{0}+\zt,\quad
\zt:=(\cM-\lambda)^{-1}\tilde{h}\in D(\cM),\quad
z_{0}:=-\lambda^{-1}h_{0}\in N(\M).$$
Then $(\M-\lambda)z=h_{0}+\tilde{h}=h$ and
$z$ depends continuously on $h$, i.e., 
$$\norm{z}_{D(\M)}
\leq\norm{z_{0}}_{D(\M)}+\norm{\zt}_{D(\M)}
=\normz{z_{0}}+\norm{\zt}_{D(\cM)}
\leq c\big(\normz{h_{0}}+\normz{\tilde{h}}\big)
\leq c\normz{h}.$$
Therefore, $\lambda\in\resol(\M)$.
We note that the inverse $(\M-\lambda)^{-1}:\Z\to D(\M)$ is given by
$$(\M-\lambda)^{-1}\pi-\lambda^{-1}\pi_{0}.$$

(i): Since (ii) implies $\spec(\M)\setminus\{0\}=\spec(\cM)\setminus\{0\}$
we just have to show the assertion for the point spectrum.

$\Rightarrow$: Let $0\neq\lambda\in\specp(\M)$. 
For $0\neq z\in N(\M-\lambda)$ we have $\lambda z=\M z\in R(\M)$.
Hence, $z\in D(\cM)$ and thus $z\in N(\cM-\lambda)$, i.e., $\lambda\in\specp(\cM)$. 

$\Leftarrow$: Of course $N(\cM-\lambda)\subset N(\M-\lambda)$.
Thus, $\lambda\in\specp(\cM)$ implies $\lambda\in\specp(\M)$. 

(iii): For $\lambda\neq0$ we have
\begin{align*}
\lambda&\in\specp(\cM)&
&\equi&
\exists\,0\neq z&\in N(\cM-\lambda)\\
&&
&\equi&
\exists\,0\neq z&\in D(\cM)\quad
\M z=\lambda z\in R(\M)\\
&&
&\equi&
\exists\,0\neq z&\in R(\M)\quad
\cM^{-1}z=\lambda^{-1}\cM^{-1}\cM z=\lambda^{-1}z\in D(\cM)\\
&&
&\equi&
\exists\,0\neq z&\in N(\cM^{-1}-\lambda^{-1})\\
&&
&\equi&
\lambda^{-1}&\in\specp(\cM^{-1}).
\end{align*}
The proof is complete.
\end{proof}

The latter lemma holds true for $\cAs\cA$ and $\cA\cAs$ we well.
More precisely:

\begin{lem}
\mylabel{spectralemAsA}
It holds
\begin{itemize}
\item[\bf(i)] 
$\spec(\As\A)\setminus\{0\}=\spec(\cAs\cA)\setminus\{0\}$, 
more precisely even
$\specc(\As\A)\setminus\{0\}=\specc(\cAs\cA)\setminus\{0\}$ and
$\specp(\As\A)\setminus\{0\}=\specp(\cAs\cA)\setminus\{0\}$,
\item[\bf(ii)] 
$\resol(\As\A)\setminus\{0\}=\resol(\cAs\cA)\setminus\{0\}$,
\item[\bf(iii)] 
$\ds\specp(\big(\cAs\cA)^{-1}\big)\setminus\{0\}=\frac{1}{\specp(\cAs\cA)\setminus\{0\}}$
and $N(\cAs\cA-\lambda^2)=N\big((\cAs\cA)^{-1}-\lambda^{-2}\big)$
for $\lambda\neq0$.
\end{itemize}
The corresponding assertions are valid for $\A\As$ and $\cA\cAs$ as well.
\end{lem}

\begin{proof}
With Lemma \ref{spectralem} (ii'), Lemma \ref{spectralemM} (i)
and Lemma \ref{lemAsAlemspec}
we have for $\lambda\neq0$
$$\lambda^2\in\spec(\As\A)\qequi
\lambda\in\spec(\M)\qequi
\lambda\in\spec(\cM)\qequi
\lambda^2\in\spec(\cAs\cA).$$
and the corresponding results hold for $\specp$, $\specc$ and $\resol$ as well. 
This shows (i) and (ii). To prove (iii) we can follow 
the proof of Lemma \ref{spectralemM} (iii) and see for $\lambda\neq0$
\begin{align*}
\lambda^2&\in\specp(\cAs\cA)&
&\equi&
\exists\,0\neq x&\in N(\cAs\cA-\lambda^2)\\
&&
&\equi&
\exists\,0\neq x&\in D(\cAs\cA)\quad
\As\A x=\lambda^2x\in R(\As)\\
&&
&\equi&
\exists\,0\neq x&\in D(\cAs\cA)\quad
\A x=\lambda^2(\cAs)^{-1}x\in R(\A)\\
&&
&\equi&
\exists\,0\neq x&\in D(\cAs\cA)\quad
x=\lambda^2(\cA)^{-1}(\cAs)^{-1}x\in R(\As)\\
&&
&\equi&
\exists\,0\neq x&\in R(\As)\quad
(\cAs\cA)^{-1}x=\lambda^{-2}x\in D(\cAs\cA)\\
&&
&\equi&
\exists\,0\neq x&\in N\big((\cAs\cA)^{-1}-\lambda^{-2}\big)\\
&&
&\equi&
\lambda^{-2}&\in\specp\big((\cAs\cA)^{-1}\big),
\end{align*}
which completes the proof.
\end{proof}

\subsubsection{Results for Compact Resolvents}

From now on we assume generally that the embedding
\begin{align}
\mylabel{compactembedding}
D(\cA)\hookrightarrow\X
\end{align}
is compact.

\begin{lem}
\mylabel{faAAs}
The following assertions hold:
\begin{itemize}
\item[\bf(i)] 
$\exists\,\ca>0\quad
\forall\,x\in D(\cA)\quad
\normx{x}\leq\ca\normy{\A x}$
\item[\bf(i')] 
$\exists\,\cas>0\quad
\forall\,y\in D(\cAs)\quad
\normy{y}\leq\cas\normx{\As y}$
\item[\bf(i'')] 
$\exists\,\cm>0\quad
\forall\,z\in D(\cM)\quad
\normz{z}\leq\cm\normz{\M z}$
\item[\bf(ii)] 
$R(\A)$, $R(\As)$ and $R(\M)$ are closed.
\item[\bf(iii)] 
$\X=N(\A)\oplus R(\As)$, $\Y=N(\As)\oplus R(\A)$ and $\Z=N(\M)\oplus R(\M)$.
\item[\bf(iv)] 
$\cA^{-1}:R(\A)\to D(\cA)$ is continuous and
$\cA^{-1}:R(\A)\to R(\As)$ is compact.
\item[\bf(iv')] 
$(\cAs)^{-1}:R(\As)\to D(\cAs)$ is continuous and
$(\cAs)^{-1}:R(\As)\to R(\A)$ is compact.
\item[\bf(iv'')] 
$\cM^{-1}:R(\M)\to D(\cM)$ is continuous and
$\cM^{-1}:R(\M)\to R(\M)$ is compact.
\item[\bf(v)] 
$D(\cAs)\hookrightarrow\Y$ is compact.
\item[\bf(v')] 
$D(\cM)\hookrightarrow\Z$ is compact.
\end{itemize}
\end{lem}

\begin{proof}
(i):
Let us assume that the estimate is wrong.
Then there exists a sequence $(x_{n})\subset D(\cA)$
with $\normx{x_{n}}=1$ and $\normy{\A x_{n}}\to0$. 
As $(x_{n})$ is bounded in $D(\cA)$, by the general assumption \eqref{compactembedding} 
we can extract a subsequence, again denoted by $(x_{n})$,
with $x_{n}\to x\in\X$. 
Since $\A$ and $\ol{R(\As)}$ are closed, we have $x\in N(\A)\cap N(\A)^{\bot}=\{0\}$,
in contradiction to $1=\normx{x_{n}}\to\normx{x}=0$.

(ii):
For $y\in\ol{R(\A)}=\ol{R(\cA)}$ there exists a sequence $(x_{n})\subset D(\cA)$ 
with $\A x_{n}\to y$. By (i') $(x_{n})$ is a Cauchy sequence in $\X$.
Hence, $(x_{n})$ converges to some $x\in\X$. 
Since $\A$ is closed, we obtain $x\in D(\A)$ and $\A x=y$, 
showing that $R(\A)$ is closed.
By the closed range theorem, see e.g.~\cite[VII, 5, Theorem]{yosidabook}, 
$R(\As)$ is closed as well. Hence, also $R(\M)=R(\As)\times R(\A)$ is closed.

(iii) follows immediately by (ii).

(iv) follows directly by (i) and \eqref{compactembedding}. Indeed,
(i) is equivalent to the continuity of $\cA^{-1}$.

(v): 
Let $(y_{n})$ be a bounded sequence in $D(\cAs)$.
By (ii), $(y_{n})\in R(\A)=R(\cA)$
and hence there exists a sequence $(x_{n})\subset D(\cA)$
with $\A x_{n}=y_{n}$. By (i), $(x_{n})$ is bounded in $D(\cA)$.
By \eqref{compactembedding}, we can extract a subsequence, again denoted by $(x_{n})$,
such that $(x_{n})$ converges in $\X$.
Then, for $x_{n,m}:=x_{n}-x_{m}$ and $y_{n,m}:=y_{n}-y_{m}$ we have
$$\normy{y_{n,m}}^2
=\scpy{\A x_{n,m}}{y_{n,m}}
=\scpx{x_{n,m}}{\As y_{n,m}}
\leq c\normx{x_{n,m}}.$$
Thus, $(y_{n})$ is a Cauchy sequence in $\Y$.

(v') is clear by \eqref{compactembedding} and (v).

(i')\footnote{(i') follows also by (iv'), 
since (i') is equivalent to the continuity of $(\cAs)^{-1}$.} 
follows by (v) analogously to (i).

(i'') follows by (i) and (i').

(iv')\footnote{Another proof of (iv') is the following:
As $\cA^{-1}:R(\A)\to R(\As)$ is compact by (iv), 
so is the adjoint $(\cAs)^{-1}:R(\As)\to R(\A)$ by Schauder's theorem,
see e.g.~\cite[X, 4, Theorem]{yosidabook}.
Especially $(\cAs)^{-1}$ is bounded
and hence also $(\cAs)^{-1}:R(\As)\to D(\As)$.} 
follows by (i') and (v).

(iv) and (iv') imply (iv'').
\end{proof}

Let us recall some facts:
By Lemma \ref{faAAs} (v') for all $\lambda\in\cn$
\begin{align}
\mylabel{compembMl}
D(\cM-\lambda)\hookrightarrow\Z
\end{align}
is compact. For $\lambda\in\resol(\M)\supset\cn\setminus\rz$ we have 
$$N(\M-\lambda)=\{0\},\quad
R(\M-\lambda)=\Z,\quad
N(\cM-\lambda)=\{0\},\quad
R(\cM-\lambda)=R(\M)$$
and the boundedness 
of $(\M-\lambda)^{-1}:\Z\to D(\M)$ is equivalent to
$$\exists\,\cml>0\quad
\forall\,z\in D(\M)\quad
\normz{z}\leq\cml\normz{(\M-\lambda)z},$$
which holds for $\cM$ as well.
For $0\neq\lambda\in\spec(\M)\subset\rz$ we have 
\begin{align*}
\Z&=N(\M-\lambda)\oplus\ol{R(\M-\lambda)},&
R(\M)&=N(\cM-\lambda)\oplus\ol{R(\cM-\lambda)}.
\end{align*}

\begin{lem}
\mylabel{faAAsl}
For $\lambda\in\rz\setminus\{0\}$ the following assertions hold:
\begin{itemize}
\item[\bf(i)] 
$N(\M-\lambda)\subset R(\M)$ and 
$N(\M-\lambda)=N(\cM-\lambda)$ has finite dimension.
\item[\bf(ii)] 
$\exists\,\cml>0\quad
\forall\,z\in D(\cM)\cap N(\M-\lambda)^{\bot}\quad
\normz{z}\leq\cml\normz{(\M-\lambda)z}$
\item[\bf(iii)] 
$R(\cM-\lambda)$ is closed.
\item[\bf(iii')] 
$R(\M-\lambda)$ is closed.
\item[\bf(iii'')] 
$R(\cM-\lambda)=R(\M-\lambda)\cap R(\M)$
\item[\bf(iv)] 
$\Z=N(\M-\lambda)\oplus R(\M-\lambda)$ and 
$R(\M)=N(\cM-\lambda)\oplus R(\cM-\lambda)$.
\item[\bf(v)] 
Let $N(\cM-\lambda)=\{0\}$. Then
$(\cM-\lambda)^{-1}:R(\M)\to D(\cM)$ is continuous and
$(\cM-\lambda)^{-1}:R(\M)\to R(\M)$ is compact.
Especially $\lambda\in\resol(\cM)$.
\item[\bf(v')] 
Let $N(\M-\lambda)=\{0\}$. Then
$(\M-\lambda)^{-1}:\Z\to D(\M)$ is continuous.
Especially $\lambda\in\resol(\M)$.
\end{itemize}
Corresponding results hold for $\As\A$, $\A\As$ resp.~$\cAs\cA$, $\cA\cAs$ we well.
\end{lem}

\begin{proof}
It is enough to consider $0\neq\lambda\in\spec(\M)\subset\rz$.

(i): 
Of course, $N(\cM-\lambda)\subset N(\M-\lambda)$. 
For $z\in N(\M-\lambda)$ we have $\M z=\lambda z$.
Thus $z\in R(\M)$, i.e., $z\in D(\cM)$. Hence $z\in N(\cM-\lambda)$.
By \eqref{compembMl} the unit ball in $N(\cM-\lambda)$ is compact, i.e.,
$\dim N(\cM-\lambda)<\infty$.

(ii): 
If the estimate is wrong, then there exists a sequence 
$(z_{n})\subset D(\cM)\cap N(\cM-\lambda)^{\bot}$
with $\normz{z_{n}}=1$ and $\normz{(\M-\lambda)z_{n}}\to0$. 
By \eqref{compembMl}
we can extract a subsequence, again denoted by $(z_{n})$,
with $z_{n}\to z\in\Z$. Moreover, 
$\M z_{n}=(\M-\lambda)z_{n}+\lambda z_{n}\to\lambda z$.
As $\M$ and $N(\cM-\lambda)^{\bot}$ are closed, 
$z$ belongs to $N(\cM-\lambda)\cap N(\cM-\lambda)^{\bot}=\{0\}$,
in contradiction to $1=\normz{z_{n}}\to\normz{z}=0$.

(iii):
Let $h\in\ol{R(\cM-\lambda)}$. Then there exists a sequence $(z_{n})\subset D(\cM)$ 
such that $(\M-\lambda)z_{n}=:h_{n}\to h$. Decomposing 
$z_{n}=z_{n,0}+\zt_{n}\in N(\M-\lambda)\oplus\ol{R(\M-\lambda)}$
shows $(\M-\lambda)\zt_{n}=h_{n}$ and $\zt_{n}\in D(\cM)\cap N(\M-\lambda)^{\bot}$.
By (ii) $(\zt_{n})$ is a Cauchy sequence in $\Z$
converging to some $z\in\Z$. Moreover,
$\M\zt_{n}=(\M-\lambda)\zt_{n}+\lambda\zt_{n}\to h+\lambda z$.
As $\cM$ is closed, we obtain $z\in D(\cM)$ and $(\M-\lambda)z=h$, 
i.e., $h\in R(\cM-\lambda)$.

(iii'):
Let $h\in\ol{R(\M-\lambda)}$. By (i) we have
$R(\M)=N(\M-\lambda)\oplus\big(R(\M)\cap\ol{R(\M-\lambda})\big)$
and hence it holds 
\begin{align}
\mylabel{RMMLcML}
R(\M)\cap\ol{R(\M-\lambda)}=R(\cM-\lambda)
\end{align}
by (iii). Let us decompose $h=h_{0}+\tilde{h}\in N(\M)\oplus R(\M)$.
As $(\M-\lambda)h_{0}=-\lambda h_{0}\in R(\M-\lambda)$, we get 
$\tilde{h}\in R(\M)\cap\ol{R(\M-\lambda)}$. 
Hence $\tilde{h}\in R(\cM-\lambda)\subset R(\M-\lambda)$
and thus $h\in R(\M-\lambda)$.

(iii'') follows by (iii') and \eqref{RMMLcML}.

(iv) follows by (iii) and (iii').

(v): 
If $N(\M-\lambda)=\{0\}$, then $R(\M-\lambda)=\Z$ and $R(\cM-\lambda)=R(\M)$.
By (ii) $(\cM-\lambda)^{-1}:R(\M)\to D(\cM)$ is continuous, more precisely,
for $h\in R(\M)$ we have $z:=(\cM-\lambda)^{-1}h\in D(\cM)$ and hence
$\normz{z}\leq\cml\normz{h}$.

(v'): 
By (i), (v) and Lemma \ref{spectralemM} (ii) 
we get $\lambda\in\resol(\cM)\setminus\{0\}=\resol(\M)\setminus\{0\}$.
Hence, $(\M-\lambda)^{-1}:\Z\to D(\M)$ is continuous.
\end{proof}

\begin{theo}
\mylabel{ONBtheo}
$\cM$ has a pure point spectrum, which is contained in $\rz\setminus\{0\}$
and point symmetric to the origin. More precisely, 
$$-\specp(\cM)
=\specp(\cM)
=\spec(\cM)
=\spec(\M)\setminus\{0\}
=\specp(\M)\setminus\{0\}$$
and
\begin{align*}
\spec(\cM)^2
&=\specp(\cAs\cA)
=\spec(\cAs\cA)
=\spec(\As\A)\setminus\{0\}
=\specp(\As\A)\setminus\{0\}\\
&=\specp(\cA\cAs)
=\spec(\cA\cAs)
=\spec(\A\As)\setminus\{0\}
=\specp(\A\As)\setminus\{0\}
\end{align*}
as well as 
$$\resol(\cM)\ni 0\in
\begin{cases}
\specp(\M)&\text{, if }N(\M)\neq\{0\},\\
\resol(\M)&\text{, if }N(\M)=\{0\}
\end{cases}$$
hold. Moreover, there exist sequences of eigenvalues and eigenvectors 
$$(\lambda_{n})_{n\in\nz}\subset(0,\infty),\quad
(z_{n}^{\pm})_{n\in\nz}=\big((x_{n},y_{n}^{\pm})\big)_{n\in\nz}\subset D(\cM),$$
which might be finite or empty if (e.g.) $\A$ is bounded, such that the following holds:
\begin{itemize}
\item[\bf(i)]
$\spec(\cM)=(\lambda_{n})\cup(-\lambda_{n})$ and 
$\spec(\cM)^2=\spec(\cAs\cA)=\spec(\cA\cAs)=(\lambda_{n}^2)$.
\item[\bf(ii)]
$(\lambda_{n})$ is monotone increasing with
$\lambda_{n}\to\infty$, if $(\lambda_{n})$ is not finite.
\item[\bf(iii)]
$(\M\mp\lambda_{n})z_{n}^{\pm}=0$ holds for all $n$, i.e.,
$\A x_{n}=\pm\lambda_{n}y_{n}^{\pm}$ and $\As y_{n}^{\pm}=\pm\lambda_{n}x_{n}$ and thus
$z_{n}^{\pm}
=(x_{n},\pm\lambda_{n}^{-1}\A x_{n})
=(\pm\lambda_{n}^{-1}\As y_{n}^{\pm},y_{n}^{\pm})$.
\item[\bf(iii')]
$(\M^2-\lambda_{n}^2)z_{n}^{\pm}=0$ holds for all $n$, i.e.,
$\As\A x_{n}=\lambda_{n}^2x_{n}$ and $\A\As y_{n}^{\pm}=\lambda_{n}^2y_{n}^{\pm}$.
\item[\bf(iv)]
$(x_{n})$ is a complete orthonormal system in $R(\As)$, i.e., 
\begin{align*}
\forall\,x&\in R(\As)&
x&=\sum_{n=1}^{\infty}\xi_{n}x_{n},
\intertext{and furthermore}
\forall\,\xt&=x_{0}+x\in\X=N(\A)\oplus R(\As)&
x&=\sum_{n=1}^{\infty}\xi_{n}x_{n},\\
\forall\,\xt&=x_{0}+x\in D(\A)=N(\A)\oplus D(\cA)&
\A\xt=\A x&=\pm\sum_{n=1}^{\infty}\lambda_{n}\xi_{n}y_{n}^{\pm},\\
\forall\,x&\in D(\As\A)&
\As\A x&=\sum_{n=1}^{\infty}\lambda_{n}^2\xi_{n}x_{n},
\end{align*}
where $\xi_{n}=\scpx{x}{x_{n}}=\scpx{\xt}{x_{n}}$. 
Moreover, $\normx{\xt}^2=\normx{x_{0}}^2+\normx{x}^2$ and
$$\normx{x}^2=\sum_{n=1}^{\infty}\xi_{n}^2,\quad
\normy{\A x}^2=\sum_{n=1}^{\infty}\lambda_{n}^2\xi_{n}^2,\quad
\normx{\As\A x}^2=\sum_{n=1}^{\infty}\lambda_{n}^4\xi_{n}^2.$$
\item[\bf(iv')]
$(y_{n}^{\pm})$ is a complete orthonormal system in $R(\A)$, i.e., 
\begin{align*}
\forall\,y&\in R(\A)&
y&=\sum_{n=1}^{\infty}\zeta_{n}^{\pm}y_{n}^{\pm},
\intertext{and furthermore}
\forall\,\yt&=y_{0}+y\in\Y=N(\As)\oplus R(\A)&
y&=\sum_{n=1}^{\infty}\zeta_{n}^{\pm}y_{n}^{\pm},\\
\forall\,\yt&=y_{0}+y\in D(\As)=N(\As)\oplus D(\cAs)&
\As\yt=\As y&=\pm\sum_{n=1}^{\infty}\lambda_{n}\zeta_{n}^{\pm}x_{n},\\
\forall\,y&\in D(\A\As)&
\A\As y&=\sum_{n=1}^{\infty}\lambda_{n}^2\zeta_{n}^{\pm}y_{n}^{\pm},
\end{align*}
where $\zeta_{n}^{\pm}=\scpy{y}{y_{n}^{\pm}}=\scpy{\yt}{y_{n}^{\pm}}$.
Moreover, $\normy{\yt}^2=\normy{y_{0}}^2+\normy{y}^2$ and
$$\normy{y}^2=\sum_{n=1}^{\infty}(\zeta_{n}^{\pm})^2,\quad
\normx{\As y}^2=\sum_{n=1}^{\infty}\lambda_{n}^2(\zeta_{n}^{\pm})^2,\quad
\normy{\A\As y}^2=\sum_{n=1}^{\infty}\lambda_{n}^4(\zeta_{n}^{\pm})^2.$$
\end{itemize}
\end{theo}

\begin{proof}
By Lemma \ref{faAAs} (iv'') we have $0\in\resol(\cM)$.
$\M=\cM$ holds if $N(\M)=\{0\}$.
By Lemma \ref{faAAsl} (v) $\cM$ has a pure point spectrum
and by Lemma \ref{faAAsl} (v') $\spec(\M)\setminus\{0\}=\specp(\M)\setminus\{0\}$.
By Lemma \ref{spectralemM} (i) we have 
$\specp(\cM)=\specp(\cM)\setminus\{0\}=\specp(\M)\setminus\{0\}$.
By Lemma \ref{spectralem} (v) the spectra are point symmetric to the origin.
The other assertions about the spectra follow immediately by 
Lemmas \ref{spectralem}, \ref{spectralemM}, \ref{spectralemAsA} 
and Lemma \ref{lemAsAlemspec}.

As $\cA^{-1}:R(\A)\to R(\As)$ or $(\cAs)^{-1}:R(\As)\to R(\A)$
are compact by Lemma \ref{faAAs} (iv) or (iv'),
so is e.g.~$(\cAs\cA)^{-1}:R(\As)\to R(\As)$. 
Moreover, $(\cAs\cA)^{-1}$ is self-adjoint and positive.
Let us assume that $\A$ is unbounded\footnote{If $\A$ is bounded,
the sequences $(\lambda_{n})$ and $(z_{n}^{\pm})$ are finite.}.
By the spectral theorem for self-adjoint, compact and non-negative operators
there exists a monotone decreasing sequence $(\lambda_{n}^{-1})_{n\in\nz}\subset(0,\infty)$ 
converging to zero and a sequence $(x_{n})_{n\in\nz}\subset R(\As)$, 
such that $\lambda_{n}^{-2}$ is an eigenvalue to the eigenvector
$x_{n}$ of $(\cAs\cA)^{-1}$, i.e., $(\cAs\cA)^{-1}x_{n}=\lambda_{n}^{-2}x_{n}$. 
Moreover, $(x_{n})$ is a complete orthonormal system in $R(\As)$,
i.e., for all $x\in R(\As)$ we have
$$x=\sum_{n=1}^{\infty}\xi_{n}(x)x_{n},\quad
\xi_{n}(x)=\scpx{x}{x_{n}}.$$
$(x_{n})\subset D(\cAs\cA)$ is also 
a complete orthonormal system of eigenvectors of $\cAs\cA$
since $\As\A x_{n}=\lambda_{n}^{2}x_{n}$.
Defining 
$$y_{n}^{\pm}:=\pm\lambda_{n}^{-1}\A x_{n}\in D(\cAs)$$ 
we see
$\As y_{n}^{\pm}=\pm\lambda_{n}x_{n}\in D(\cA)$.
Hence, $y_{n}^{\pm}\in D(\cA\cAs)$ with
$\A\As y_{n}^{\pm}=\pm\lambda_{n}\A x_{n}=\lambda_{n}^{2}y_{n}^{\pm}$, i.e.,
$y_{n}^{\pm}$ is an eigenvector of $\cA\cAs$ to the eigenvalue $\lambda_{n}^{2}$.
For all $y\in R(\A)$ with $y=\A x$ for some $x\in D(\A)$ we have
\begin{align}
\mylabel{ynxnform}
\scpy{y}{y_{n}^{\pm}}
=\scpx{x}{\As y_{n}^{\pm}}
=\pm\lambda_{n}\scpx{x}{x_{n}}.
\end{align}
This shows two things. First, 
putting $y:=y_{m}^{\pm}=\A(\pm\lambda_{m}^{-1}x_{m})$ we get
$$\scpy{y_{m}^{\pm}}{y_{n}^{\pm}}
=\frac{\lambda_{n}}{\lambda_{m}}\scpx{x_{m}}{x_{n}},$$
which shows that $(y_{n}^{+})$ and $(y_{n}^{-})$ are both orthonormal systems in $R(\A)$,
and second, that they are even complete in $R(\A)$.
Thus, for all $y\in R(\A)$ we obtain
$$y=\sum_{n=1}^{\infty}\zeta_{n}^{\pm}(y)y_{n}^{\pm},\quad
\zeta_{n}^{\pm}(y)=\scpy{y}{y_{n}^{\pm}}.$$
A little more careful inspection shows the following:
For all $y=\A x\in R(\A)$ with $x\in D(\A)$ we have again with \eqref{ynxnform}
\begin{align*}
y&=\sum_{n=1}^{\infty}\zeta_{n}^{\pm}(y)y_{n}^{\pm}
=\pm\sum_{n=1}^{\infty}\lambda_{n}\xi_{n}(x)y_{n}^{\pm}
=\sum_{n=1}^{\infty}\xi_{n}(x)\A x_{n},\\
\zeta_{n}^{\pm}(y)
&=\scpy{y}{y_{n}^{\pm}}
=\pm\lambda_{n}\scpx{x}{x_{n}}
=\pm\lambda_{n}\xi_{n}(x).
\intertext{If even $y=\A\As\yt\in R(\A\As)$ with $\yt\in D(\A\As)$ we see}
y&=\sum_{n=1}^{\infty}\zeta_{n}^{\pm}(y)y_{n}^{\pm}
=\sum_{n=1}^{\infty}\lambda_{n}^2\zeta_{n}^{\pm}(\yt)y_{n}^{\pm}
=\sum_{n=1}^{\infty}\zeta_{n}^{\pm}(\yt)\A\As y_{n}^{\pm},\\
\zeta_{n}^{\pm}(y)
&=\scpy{\yt}{\A\As y_{n}^{\pm}}
=\lambda_{n}^2\scpy{\yt}{y_{n}^{\pm}}
=\lambda_{n}^2\zeta_{n}^{\pm}(\yt).
\intertext{Analogously for some $x=\As y\in R(\As)$ with $y\in D(\As)$ it holds}
x&=\sum_{n=1}^{\infty}\xi_{n}(x)x_{n}
=\pm\sum_{n=1}^{\infty}\lambda_{n}\zeta_{n}^{\pm}(y)x_{n}
=\sum_{n=1}^{\infty}\zeta_{n}^{\pm}(y)\As y_{n}^{\pm},\\
\xi_{n}(x)
&=\scpx{x}{x_{n}}
=\scpy{y}{\A x_{n}}
=\pm\lambda_{n}\scpy{y}{y_{n}^{\pm}}
=\pm\lambda_{n}\zeta_{n}^{\pm}(y).
\intertext{If even $x=\As\A\xt\in R(\As\A)$ with $\xt\in D(\As\A)$ we have}
x&=\sum_{n=1}^{\infty}\xi_{n}(x)x_{n}
=\sum_{n=1}^{\infty}\lambda_{n}^2\xi_{n}(\xt)x_{n}
=\sum_{n=1}^{\infty}\xi_{n}(\xt)\As\A x_{n},\\
\xi_{n}(x)
&=\scpx{\xt}{\As\A x_{n}}
=\lambda_{n}^2\scpx{\xt}{x_{n}}
=\lambda_{n}^2\xi_{n}(\xt).
\end{align*}
For $z_{n}^{\pm}:=(x_{n},y_{n}^{\pm})\in D(\cM)$ we have
$$\M z_{n}^{\pm}
=(\As y_{n}^{\pm},\A x_{n})
=\pm\lambda_{n}(x_{n},y_{n}^{\pm})
=\pm\lambda_{n}z_{n}^{\pm}.$$
Hence, $z_{n}^{\pm}$ is an eigenvector 
to the eigenvalue $\pm\lambda_{n}$ of $\M$, i.e., 
$z_{n}^{\pm}\in N(\M\mp\lambda_{n})$.
Of course, $z_{n}^{\pm}$ is also an eigenvector
to the eigenvalue $\lambda_{n}^2$ of $\M^2$ since
$$\zmat{\As\A-\lambda_{n}^2}{0}{0}{\A\As-\lambda_{n}^2}
=\M^2-\lambda_{n}^2=(\M\pm\lambda_{n})(\M\mp\lambda_{n}).$$
The assertions about the norms follow immediately 
by orthogonality and the continuity of the norms,
concluding the proof.
\end{proof}

\begin{cor}
\mylabel{eigenvaluesmin}
It holds
$$\lambda_{\ell}^2
=\normy{\A x_{\ell}}^2
=\min_{\substack{0\neq x\in D(\cA)\\x\bot_{X}\{x_{1},\dots,x_{\ell-1}\}}}
\frac{\normy{\A x}^2}{\normx{x}^2}
=\min_{\substack{0\neq y\in D(\cAs)\\y\bot_{Y}\{y_{1}^{\pm},\dots,y_{\ell-1}^{\pm}\}}}
\frac{\normx{\As y}^2}{\normy{y}^2}
=\normx{\As y_{\ell}^{\pm}}^2,$$
especially 
$$\lambda_{1}^2
=\min_{0\neq x\in D(\cA)}\frac{\normy{\A x}^2}{\normx{x}^2}
=\min_{0\neq y\in D(\cAs)}\frac{\normx{\As y}^2}{\normy{y}^2}.$$
\end{cor}

\begin{proof}
First, we emphasize that the dimensions of the eigenspaces 
$N(\As\A-\lambda_{n}^2)$ and $N(\A\As-\lambda_{n}^2)$ equal.
Using the latter theorem we can represent $x\in D(\cA)$ and $\A x$ by
$$x=\sum_{n=1}^{\infty}\xi_{n}x_{n},\quad
\A x=\pm\sum_{n=1}^{\infty}\lambda_{n}\xi_{n}y_{n}^{\pm}.$$
If additionally $x\bot_{X}\{x_{1},\dots,x_{\ell-1}\}$ we see $\xi_{1}=\dots=\xi_{\ell-1}=0$ and thus
$$\normx{x}^2=\sum_{n=\ell}^{\infty}\xi_{n}^2,\quad
\normy{\A x}^2=\sum_{n=\ell}^{\infty}\lambda_{n}^2\xi_{n}^2
\geq\lambda_{\ell}^2\sum_{n=\ell}^{\infty}\xi_{n}^2
=\lambda_{\ell}^2\normx{x}^2.$$
Therefore, $\ds\frac{\normy{\A x}^2}{\normx{x}^2}\geq\lambda_{\ell}^2$ holds for
all $0\neq x\in D(\cA)$ with $x\bot_{X}\{x_{1},\dots,x_{\ell-1}\}$. On the other hand 
$\normy{\A x_{\ell}}^2=\scpx{x_{\ell}}{\As\A x_{\ell}}=\lambda_{\ell}^2\normx{x_{\ell}}^2$
and $0\neq x_{\ell}\in D(\cA)$ with $x_{\ell}\bot_{X}\{x_{1},\dots,x_{\ell-1}\}$. Thus,
$$\lambda_{\ell}^2
=\normy{\A x_{\ell}}^2
=\min_{\substack{0\neq x\in D(\cA)\\x\bot_{X}\{x_{1},\dots,x_{\ell-1}\}}}
\frac{\normy{\A x}^2}{\normx{x}^2}.$$
The other assertion about $y$ and $\As y$ follows analogously.
\end{proof}

\end{document}

%% file: head_paper.tex
\usepackage{a4,exscale,ifthen,amsfonts,amssymb,amsmath,amscd,graphicx}
\usepackage[english]{babel}
\usepackage[mathscr]{eucal}

\author{{\sf\pauthor}}
\numberwithin{equation}{section}

\newenvironment{acknow}{{\vspace*{1cm}\noindent\bf Acknowledgements }}{}
\newcommand{\bewboxw}{\mbox{}\hfill $\square$ \\}

\newenvironment{proof}{{\noindent\bf Proof }}{\bewboxw}

\newcommand{\keywords}[1]{{\noindent\bf Key Words }#1}

\ifthenelse{\equal{\mylabelonoff}{on}}
{\newcommand{\mylabel}[1]{\label{#1}\fbox{{\rm #1}}}}{\newcommand{\mylabel}[1]{\label{#1}\makebox[0mm][]{}}}
\ifthenelse{\equal{\allowdisbrk}{yes}}
{\allowdisplaybreaks}{}

%% file: head_paule.tex
\usepackage[mathscr]{eucal}
\usepackage{color}

\newcommand{\ds}{\displaystyle}
\newcommand{\ol}{\overline}
\newcommand{\ul}{\underline}

\newcommand{\nz}{\mathbb{N}}

\newcommand{\rz}{\mathbb{R}}

\newcommand{\rt}{\rz^3}
\newcommand{\rttt}{\rz^{3\times3}}

\newcommand{\rN}{\rz^N}

\DeclareMathOperator{\p}{\partial}

\newcommand{\na}{\nabla}

\DeclareMathOperator{\ed}{d}

\DeclareMathOperator{\cd}{\delta}

\DeclareMathOperator{\rot}{rot}

\renewcommand{\div}{\operatorname{div}}

\newcommand{\T}{T}

\newcommand{\eps}{\varepsilon}
\newcommand{\epsmo}{\eps^{-1}}
\newcommand{\mumo}{\mu^{-1}}

\newcommand{\epsu}{\ul{\eps}}
\newcommand{\epso}{\ol{\eps}}
\newcommand{\epsh}{\hat{\eps}}
\newcommand{\epsch}{\check{\eps}}

\newcommand{\sigmap}{\sigma_{\mathtt{p}}}
\newcommand{\om}{\Omega}

\newcommand{\pom}{\p\!\om}

\newcommand{\ga}{\Gamma}

\newcommand{\gan}{\ga_{n}}

\newcommand{\equi}{\Leftrightarrow}
\def\impl{\Rightarrow}

\newcommand{\qequi}{\quad\equi\quad}

\DeclareMathOperator{\supp}{supp}

\DeclareMathOperator{\dist}{dist}

\newcommand{\zmat}[4]{\begin{bmatrix}#1&#2\\#3&#4\end{bmatrix}}

\DeclareMathOperator{\id}{id}

\def\set#1#2{\{#1\,:\,#2\}}

\newcommand{\cp}{c_{\mathtt{p}}}

\newcommand{\cpeps}{c_{\mathtt{p},\eps}}

\newcommand{\cm}{c_{\mathtt{m}}}

\newcommand{\cmrot}{c_{\mathtt{m},\rot}}

\DeclareMathOperator{\Lebesgue}{\mathsf{L}}
\newcommand{\Lgen}[2]{\Lebesgue^{#1}_{#2}}
\def\Li{\Lgen{\infty}{}}

\def\Lt{\Lgen{2}{}}

\def\Ltom{\Lt(\om)}

\def\Ltepsom{\Lgen{2}{\eps}(\om)}
\def\Ltepsmoom{\Lgen{2}{\epsmo}(\om)}

\def\Ltmuom{\Lgen{2}{\mu}(\om)}
\def\Ltmumoom{\Lgen{2}{\mumo}(\om)}

\newcommand{\qLt}[1]{\Lgen{2,#1}{}}
\newcommand{\Ltq}{\qLt{q}}
\newcommand{\Ltqpo}{\qLt{q+1}}
\newcommand{\Ltqmo}{\qLt{q-1}}
\newcommand{\Ltqom}{\Ltq(\om)}
\newcommand{\Ltqpoom}{\Ltqpo(\om)}
\newcommand{\Ltqmoom}{\Ltqmo(\om)}

\DeclareMathOperator{\DSobolev}{\mathsf{D}}
\newcommand{\Dgen}[3]{\overset{#3}{\DSobolev}{}^{#1}_{#2}}
\newcommand{\qD}[1]{\Dgen{#1}{}{}}

\newcommand{\Dq}{\qD{q}}

\newcommand{\Dqom}{\Dq(\om)}

\DeclareMathOperator{\DeSobolev}{\Delta}
\newcommand{\Degen}[3]{\overset{#3}{\DeSobolev}{}^{#1}_{#2}}
\newcommand{\qDe}[1]{\Degen{#1}{}{}}

\newcommand{\Deq}{\qDe{q}}
\newcommand{\Deqpo}{\qDe{q+1}}

\newcommand{\Deqom}{\Deq(\om)}
\newcommand{\Deqpoom}{\Deqpo(\om)}

\DeclareMathOperator{\Sobolev}{\mathsf{H}}
\newcommand{\Hgen}[3]{\overset{#3}{\Sobolev}{}^{#1}_{#2}}
\def\Ho{\Hgen{1}{}{}}

\def\Hoom{\Ho(\om)}

\newcommand{\Cgen}[2]{\overset{#2}{\Cont}{}^{#1}}

\newcommand{\scp}[2]{\left\langle#1,#2\right\rangle}
\newcommand{\scps}[2]{\langle#1,#2\rangle}

\newcommand{\scpLtom}[2]{\scp{#1}{#2}_{\Ltom}}

\newcommand{\norm}[1]{\left|\normdst\left|#1\right|\normdst\right|}

\newtheorem{lem}{Lemma}

\newtheorem{theo}[lem]{Theorem}
\newtheorem{cor}[lem]{Corollary}
\newtheorem{rem}[lem]{Remark}

\DeclareMathOperator{\rotspace}{\mathsf{R}}
\newcommand{\rotgen}[2]{\overset{#2}{\rotspace}{}_{#1}}
\newcommand{\rotgenom}[2]{\rotgen{#1}{#2}(\om)}
\newcommand{\rom}{\rotgenom{}{}}

\newcommand{\rzom}{\rotgenom{0}{}}

\DeclareMathOperator{\divspace}{\mathsf{D}}
\newcommand{\divgen}[2]{\overset{#2}{\divspace}{}_{#1}}
\newcommand{\divgenom}[2]{\divgen{#1}{#2}(\om)}
\newcommand{\dom}{\divgenom{}{}}

\newcommand{\dzom}{\divgenom{0}{}}

\newcommand{\hoom}{\Hoom}

\DeclareMathOperator{\harmonic}{\mathcal{H}}
\newcommand{\harm}[2]{\harmonic^{#1}_{#2}}
\newcommand{\harmom}[2]{\harm{#1}{#2}(\om)}
\newcommand{\harmdnom}{\harmom{}{\mathtt{DN}}}
\newcommand{\harmndom}{\harmom{}{\mathtt{ND}}}

\newcommand{\harmdepsom}{\harmom{}{\mathtt{D},\eps}}
\newcommand{\harmdnepsom}{\harmom{}{\mathtt{DN},\eps}}
\newcommand{\harmndepsom}{\harmom{}{\mathtt{ND},\eps}}

\newcommand{\harmnepsom}{\harmom{}{\mathtt{N},\eps}}

\newcommand{\pid}{\pi_{\mathtt{D}}}
\newcommand{\pin}{\pi_{\mathtt{N}}}
\newcommand{\pidn}{\pi_{\mathtt{DN}}}

\newcommand{\normos}[1]{|#1|} 

\newcommand{\normosom}[1]{\normos{#1}_{\om}} 

\DeclareMathOperator{\diam}{diam}